\documentclass{amsart}
\usepackage{amssymb, amsfonts, lacromay}
\usepackage{mathrsfs,comment}
\usepackage[usenames,dvipsnames]{color}
\usepackage[normalem]{ulem}
\usepackage{url}
\usepackage{amsmath}
\usepackage{amsthm}
\usepackage{amssymb}
\usepackage{graphicx}
\usepackage{stmaryrd}
\usepackage[colorlinks=true, allcolors=blue]{hyperref}
\usepackage{quiver}
\usepackage{cleveref}
\usepackage{mathrsfs}
\usepackage{enumerate}

\usepackage[all,arc,2cell]{xy}
\UseAllTwocells

\usepackage{enumerate}

\usepackage{todonotes}
\usepackage{hyperref}  
\hypersetup{%
  bookmarksnumbered=true,%
  bookmarks=true,%
  colorlinks=true,%1
  linkcolor=blue,%
  citecolor=blue,%
  filecolor=blue,%
  menucolor=blue,%
  pagecolor=blue,%
  urlcolor=blue,%
  pdfnewwindow=trueƒ,%
  pdfstartview=FitBH}   
  
\usepackage[inline]{showlabels}

\def\makeautorefname#1#2{\expandafter\def\csname#1autorefname\endcsname{#2}}
\def\equationautorefname~#1\null{(#1)\null}
\makeautorefname{footnote}{footnote}%
\makeautorefname{item}{item}%
\makeautorefname{figure}{Figure}%
\makeautorefname{table}{Table}%
\makeautorefname{part}{Part}%
\makeautorefname{appendix}{Appendix}%
\makeautorefname{chapter}{Chapter}%
\makeautorefname{section}{Section}%
\makeautorefname{subsection}{Section}%
\makeautorefname{subsubsection}{Section}%
\makeautorefname{theorem}{Theorem}%
\makeautorefname{thm}{Theorem}%
\makeautorefname{sta}{Statement}%
\makeautorefname{cor}{Corollary}%
\makeautorefname{lem}{Lemma}%
\makeautorefname{prop}{Proposition}%
\makeautorefname{pro}{Property}
\makeautorefname{conj}{Conjecture}%
\makeautorefname{defn}{Definition}%
\makeautorefname{notn}{Notation}
\makeautorefname{notns}{Notations}
\makeautorefname{rem}{Remark}%
\makeautorefname{rmk}{Remark}%
\makeautorefname{rems}{Remarks}%
\makeautorefname{quest}{Question}%
\makeautorefname{exmp}{Example}%
\makeautorefname{ax}{Axiom}%
\makeautorefname{claim}{Claim}%
\makeautorefname{assn}{Assumption}%
\makeautorefname{asses}{Assumptions}%
\makeautorefname{con}{Construction}%
\makeautorefname{prob}{Problem}%
\makeautorefname{proj}{Project}%
\makeautorefname{warn}{Warning}%
\makeautorefname{obs}{Observation}%
\makeautorefname{conv}{Convention}%
\makeautorefname{cont}{Context}%

\newtheorem{thm}{Theorem}[section]

\newtheorem{prop}{Proposition}[section]
\newtheorem{lem}{Lemma}[section]

\theoremstyle{definition}
\newtheorem{defn}{Definition}[section]

\newtheorem{rem}{Remark}[section]

\newtheorem{warn}{Warning}[section]
\newtheorem{sch}{Scholium}[section]

\newcounter{assn}[section]
\renewcommand{\theassn}{\Alph{assn}}

\makeatletter
\let\c@cont=\c@thm
\let\c@conv=\c@thm
\let\c@obs=\c@thm
\let\c@sta=\c@thm
\let\c@cor=\c@thm
\let\c@prop=\c@thm
\let\c@lem=\c@thm
\let\c@prob=\c@thm
\let\c@con=\c@thm
\let\c@conj=\c@thm
\let\c@defn=\c@thm
\let\c@notn=\c@thm
\let\c@notns=\c@thm
\let\c@exmp=\c@thm
\let\c@pro=\c@thm
\let\c@ass=\c@thm
\let\c@warn=\c@thm
\let\c@rem=\c@thm
\let\c@sch=\c@thm
\let\c@equation\c@thm
\numberwithin{equation}{section}
\makeatother

\definecolor{orange}{rgb}{1,0.5,0}

\newcommand{\bk}{\mathbf{k}}
\newcommand{\bm}{\mathbf{m}}
\newcommand{\bn}{\mathbf{n}}
\newcommand{\bp}{\mathbf{p}}
\newcommand{\bq}{\mathbf{q}}

\newcommand{\PPs}{\sP\text{-}\mathbf{PsAlg}}
\newcommand{\FP}{\sF\text{-}\mathbf{PsAlg}}
\newcommand{\FPs}{\sF_{vs}\text{-}\mathbf{PsAlg}}
\newcommand{\FPR}{\sF_{\bR}\text{-}\mathbf{PsAlg}}
\newcommand{\SM}{\mathbf{SymMon}}


\title{What are symmetric monoidal categories?}

\author{Jiasen Liu}
\address{University of Southern California}
\email{jiasen.jason.liu@gmail.com}

\author{J. P. May}
\address{Department of Mathematics, The University of Chicago, Chicago, IL 60637}
\email{may@math.uchicago.edu}

\author{Kyle I Roke}
\address{MIT}
\email{kroke@mit.edu}

\author{Hongyi Zhang}
\address{Haverford}
\email{hzhang3@haverford.edu}

\author{Keming Zhou}
\address{Wellesley}
\email{kz105@wellesley.edu}

\begin{document}

\begin{abstract} 
Symmetric monoidal categories have been understood since the 1960's and are central to many branches of mathematics.  In particular, the construction of spectra from symmetric monoidal categories is at the heart of algebraic $K$-theory.  This construction starts from either categories with an action by a suitable operad $\sP$ or with suitable functors from the category  
$\sF$ of finite sets to the category $\mathbf{Cat}$ of categories.  Infinite loop space theory, which codifies these constructions, led to the invention of $\infty$-categories.  So why the title?  We shall prove that the $2$-category $\SM$ of symmetric monoidal categories is equivalent (in fact very nearly isomorphic) both to a $2$-category $\PPs$ of $\sP$-pseudoalgebras and to an isomorphic $2$-category $\FPR$ of strictly special $\sF$-pseudofunctors.   This equivalence underlies a streamlined equivariant and multiplicative enhancement of infinite loop space theory and should be of independent interest.
\end{abstract}

\maketitle

\tableofcontents

\section*{Introduction} It has been understood since the early 1970's that permutative categories and suitable $\sF$-categories naturally give rise to spectra.  Here $\sF$  is the standard skeleton of the category of finite sets, the opposite of Segal's category $\GA$ \cite{MayPerm, Seg}.  Permutative categories are defined to be symmetric {\em strict} monoidal categories and are equivalent to $\sP$-algebras, where $\sP$ is the permutativity operad (see \autoref{PPs}).   It was also understood in the 1970's that Mac Lane's coherence theorem \cite{Mac}  for symmetric monoidal categories can be recast as the assertion that symmetric monoidal categories can be functorially strictified to equivalent permutative categories.  Although written proofs are recent, due independently to Yau \cite{Yau2} and us \cite{LRZZ}, it was folklore in the 1970's that symmetric monoidal categories are more directly equivalent to $\sP$-pseudoalgebras, as we shall rework below.   What is new is that $\sP$-pseudoalgebras are also directly {\em{ isomorphic}} to strictly special $\sF$-pseudoalgebras, as defined in  Definitions \ref{PIspec} and  \ref{Fps}.

The intuition that this should be true came to the senior author just before Chicago's 2025 REU began. He farmed out the details to his four coauthors, who were participants in that REU.   His motivation is and was that this should be the starting point of a cleaner and more powerful approach to equivariant multiplicative infinite loop space theory than appears in the literature.  The  point is to construct $E_{\infty}$-ring $G$-spectra and related structures from categorical input.   That intuition is correct, but the details are not yet written down.   Since the starting point is of independent interest, we have completed the details of the preliminary drafts \cite{LRZZ, LRZZ2} for presentation here.

Modulo definitions, the two main theorems are easily stated.

\begin{thm}\label{thm1}  The $2$-category $\SM$  is equivalent to the $2$-category $\PPs$.
\end{thm}

\begin{thm}\label{thm2}  The $2$-category $\PPs$ is equivalent to the $2$-category $\FPs$ and it is isomorphic to the $2$-category $\FPR$. 
\end{thm}

In \autoref{Sec1}, we describe  all of the $2$-categories named in these theorems.  In Sections \ref{Sec2} and \ref{Sec3}, we describe the comparison functors and sketch the proofs of the theorems. In \autoref{Sec4}, we fill in the diagrams and categorical details needed to substantiate the sketches.   The point is that the intuition is clear but the details are categorical coherence theory that is easy to see but not so easy to write out in detail.   Experts will imagine that the combinatorics are best hidden in terms of a reformulation of results in terms of pseudoalgebras over $2$-monads, but we explain our lack of progress with such a reinterpretation in \autoref{Sec5}.

\section{Intuitive definitions and ideas}\label{Sec1}

\subsection{The $2$-category $\SM$}\label{SMC1}

We shall recall the full definition of the $2$-category $\SM$ of symmetric monoidal categories in \autoref{SMdetails}.  Briefly, a symmetric monoidal category $\sA = (\sA,\otimes, e)$ consists of a category $\sA$, a  product $\otimes\colon \sA\times \sA \rtarr \sA$, and a unit object $e$ such that $\otimes$ is associative, unital, and commutative up to coherent natural isomorphisms.  The diagrams  giving coherence are given in \autoref{SMObj}.  Mac Lane's original coherence theorem is stated in \autoref{SMCoh}.  
                                                                                                                                                                                                                                                                                                                                                                                                                                                                           
The category of morphisms between symmetric monoidal categories $\sA$ and $\sB$  has objects the symmetric monoidal functors $\sA \rtarr \sB$  and morphisms the natural transformations of such functors.   A functor $\bF\colon \sA \rtarr \sB$ is symmetric monoidal if there are unit and product natural isomorphisms that make associativity, unitality, and symmetry diagrams commute, as is made precise in \autoref{SMMap}.  The natural transformations must commute with the unit and product isomorphisms, as is made precise in \autoref{SMMap2}.    Examples appear naturally throughout mathematics.

\subsection{The $2$-category $\PPs$}\label{PPs}
We will not repeat the full definition of an operad from \cite{MayGeo}, \cite{MayOp1}, \cite{GMMO1} or \cite{LRZZ2}.  We will only be concerned with operads in $\mathbf{Cat}$ and with one in particular, namely the permutativity operad $\sP$. There is an operad $\sM$  (alias $\mathbf{Assoc}$) in the category $\mathbf{Set}$ of sets. Its  $j$th set is the symmetric group $\SI_j$.  The right adjoint  to the object functor $\mathbf{Cat} \rtarr \mathbf{Set}$ is the chaotic or indiscrete category functor $\mathcal{E}$.  It takes a set $S$ to the category (a groupoid) with objects $S$ and  exactly one arrow between any pair of elements of $S$.  The operad $\sP$ is $\mathcal{E}\sM$. %It is reduced, both $\sP(0)$ and $\sP(1)$ being copies of the trivial category $\mathbf{1}$.  

Connor and Gurski \cite{CG} introduced the notion of a pseudoalgebra over an operad, and largely the same definition was at the starting point of \cite{GMMO1}, to which this paper can be viewed as a sequel.  Restricting to $\sP$, we will recall the definition in \autoref{PPsDefn}.  Briefly, a $\sP$-pseudoalgebra $\sA$ comes with $\SI_j$-equivariant functors  $\tha_j\colon \sP(j)\times \sA^j \rtarr \sA$  such that the unit and composition diagrams that are required to commute for a $\sP$-algebra $\sA$ now commute only up to equivariant and coherent natural isomorphisms.  The morphisms and $2$-morphisms between pseudoalgebras are required to be compatible with these isomorphisms.  These notions give the $2$-category $\PPs$ of $\sP$-pseudoalgebras.  While the definition may not be familiar, in our case of $\sP$ it just encodes exactly what we see in the familiar notion of the $2$-category of symmetric monoidal categories. 

In fact, it has been known since \cite{MayPerm} that the category of permutative categories is isomorphic to the category of $\sP$-algebras. 
\autoref{thm1} is the folklore generalization to symmetric monoidal categories.  We will define the inverse equivalences 
$\bU$ and $\bV$ in Sections \ref{U} and \ref{V}.  The idea is that $\sP$-pseudoalgebras are ``unbiased'', in that they have a canonically given $n$-fold product $\sA^n \rtarr \sA$ for each $n$.  All variants thereof via permutations, composition, and unit conditions are built into the pseudoalgebra structure.  

Symmetric monoidal categories start with a $2$-fold product and its structural properties, and the idea is that these properties give minimal generators for a $\sP$-pseudoalgebra structure on the same category.  A $\sP$-pseudoalgebra $\sA$ has an underlying symmetric monoidal category $\bU \sA$.  For a symmetric monoidal category $\sA$, a canonical $n$-fold product is obtained by iterating its product via a chosen choice of how to apply its associativity isomorphism.  This elaborates to the $2$-functor $\bV$. As the intuition suggests, $\bU \bV$ is the identity on $\SM$. The fact that $\bV \bU$ is isomorphic to the identity on $\PPs$ is implicit in Mac Lane's coherence theorem for symmetric monoidal categories, as we shall see in \autoref{Sketch1}.

\subsection{The $2$-categories $\FPs$ and $\FPR$}\label{FPs}
We start by recalling some categories of finite sets.  The ones of greatest interest are $\sF$, $\sE$, and $\PI$; $\LA$ and $\LA^{\perp}$ will only appear briefly, in \autoref{Sec5}.

\begin{defn}\label{Fetal}  All categories in this definition have objects the based finite sets $\bn = \{0, 1, \cdots, n\}$, with basepoint $0$, and morphisms certain based functions.  The category $\sF$ has all (based) functions. Its subcategory $\PI$ has those maps 
$\ph\colon \bm \rtarr \bn$ such that $\ph^{-1}(j)$ has $0$ or $1$ element for $1\leq j\leq n$.   The subcategory  $\LA$ of $\PI$ has 
those maps such that $\ph^{-1}(0) = \{0\}$.  Thus $\LA$ is the subcategory of injections and permutations, while $\PI$ adds in projections.  Let $\LA^{\perp}$ be the subcategory of $\PI$ whose maps are those such that $\ph^{-1}(j)$ has $1$ element for $1\leq j\leq n$.  Thus $\LA^{\perp}$ is the subcategory of projections and permutations and $\LA\cap \LA^{\perp} = \coprod_n \LA(\bn,\bn) \iso \SI$, where $\SI$ is the disjoint union of the groups $\SI_n$.  Let $\sE$ denote  the subcategory of $\sF$ whose maps are  those such that $\ph^{-1}(0) = \{0\}$, so that  $\LA = \PI \cap \sE$.\footnote{Maps in $\sE$ are called effective in \cite[Notations 5.2]{MT}.}  We can view all of these categories as discrete $2$-categories, with only identity $2$-cells.
\end{defn}

\begin{defn} For $n\geq 0$, let $\ph_n\in \sE$  be that function $\bn \rtarr \mathbf{1}$ that sends $i$ to $1$ for $1\leq i\leq n$.  We call $\ph_n$ the canonical $n$-fold product.
\end{defn}

\begin{defn} Let $\PI$-$\bf{Cat}$ denote the category of functors $\sB\colon \PI \rtarr \bf{Cat}$ and natural transformations and define $\LA$-$\bf{Cat}$ and
$\LA^{\perp}$-$\bf{Cat}$ similarly.  We often write $\sB(\bf{n})$ as $\sB_n$.  We require our functors $\sB$ to be reduced, meaning that $\sB_0$ is the trivial category $\mathbf{\ast}$ with one object (also denoted $\ast$).  The unique maps $\mathbf{0} \rtarr \bn$ give the $\sB_n$ compatible base objects.  
\end{defn}

\begin{defn}\label{PIspec} For $1\leq i\leq n$, the maps $\de_i\colon \bn \rtarr \mathbf{1}$ in $\PI$ that send $i$ to $1$  and all other $j$ to $0$ induce the components of a natural map $\de\colon \sB_n \rtarr \sB_1^{n}$. We say that $\sB$ is {\em{very special}} if $\de$ is an {\em{isomorphism}} of  categories for all $n\geq 0$.  We say that $\sB$ is {\em{strictly special}} if $\de$ is the identity isomorphism for all $n\geq 0$.
\end{defn}

\begin{rem}  In analogy with Segal's definition for spaces, $\sB$ is said to be {\em{special}} if $\de$ is an {\em{equivalence}} of categories for all $n$. While we shall make relevant remarks, we shall {\em{never}} use that definition in this paper.  Starting from \cite{Arroyo}, it has been used to give an equivalence of suitably defined homotopy categories of symmetric monoidal categories and of special $\sF$-categories.  Incorporating more structure, this lifts to an equivalence of associated $\infty$-categories.  Our \autoref{thm2} is substantially more precise, and its formulation and proof require no elaborate theoretical context.
\end{rem}

The following definition and observation give the maps $\de$ a conceptual meaning.

\begin{defn} Let $\bf{Cat}_*$ be  the category of based categories.  Define a functor $\bL \colon \PI\text{-}\bf{Cat} \rtarr \bf{Cat}_*$ by $\bL \sB = \sB(1)$.
Define a functor $\bR\colon \bf{Cat}_* \rtarr \PI\text{-}\bf{Cat}$ by $(\bR \sA)(n) = \sA^n$, with the evident permutations and projections and with injections that use base objects to insert $\ast$ into powers of $\sA$. Clearly $\bR\bL = \id$.
\end{defn}

\begin{lem}\label{LRAd} The pair $(\bL,\bR)$ is an adjunction. Its counit is the identity and its unit is $\de\colon \Id \rtarr \bR\bL$. 
Therefore the adjunction restricts to an adjoint equivalence between $\bf{Cat}_{*}$ and the category of very special $\PI$-categories and to an adjoint isomorphism between $\bf{Cat}_{*}$ and the category of strictly special $\PI$-categories.
\end{lem}

Functors $\sF \rtarr \bf{Cat}$ that are very special or strictly special when restricted to $\PI$ are too strict to be of much use. However, there are standard notions of pseudofunctors between $2$-categories and of pseudotransformation  between pseudofunctors (\autoref{PsFun}).   We shall see that
$\sF$-pseudofunctors that restrict on $\PI$ to actual functors that are very special or strictly special are 
especially useful.   

\begin{defn}\label{Fps} An $\sF$-pseudoalgebra is a pseudofunctor $\sF \rtarr  \bf{Cat}$ whose restriction to $\PI$ is a functor.  It is {\em{very special}} (respectively, {\em{strictly special}}) if its restriction to $\PI$ is very special (respectively, strictly special).  An $\sF$-pseudomorphism of $\sF$-pseudoalgebras is a pseudotransformation that restricts on $\PI$ to a (strict) natural transformation. Defining maps between 
$\sF$-pseudomorphisms, which we call $\sF$-$2$-cells (\autoref{2cell2}), we obtain a $2$-category $\FP$ of $\sF$-pseudofunctors, $\sF$-pseudotransformations, and $\sF$-$2$-cells.  Define $\FPs\subset \FP$ to be the full sub $2$-category of very special $\sF$-pseudoalgebras.  Define $\FPR\subset \FPs$ to be the full sub $2$-category of strictly special $\sF$-pseudoalgebras. 
\end{defn}

\begin{rem}\label{example}
The strictness (functor, not just pseudofunctor) on restriction to $\PI$ is essential in this paper and is also vital to the preservation of symmetry in multiplicative infinite loop space theory.  Clearly we do not have such strictness when composing morphisms in $\sF$ with morphisms in $\PI$.  For a centrally important example, it is clear that  $\ph_n\si = \ph_n$ for $\si\in \SI_n\subset \sF_n$. Therefore the product $(\ph_n\si)(a)$ is not equal to the permuted product $\ph_n(\si a)$ for $a\in \sA_n$. 
\end{rem}

\begin{rem}\label{clarity} Dropping any specialness restriction, we define $\sF\text{-}\bf{AlgSt}$ to be the $1$-category of functors $\sF \rtarr \bf{Cat}$ and natural transformations between them and we then define $\sF_{s}\text{-}\bf{AlgSt}$ to be its full subcategory of special functors.  There is a strictification functor 
$\mathbf{St}\colon \FP \rtarr \sF\text{-}\bf{AlgSt}$.  Since $\mathbf{St}\sB$ is equivalent to $\sB$, it is special if $\sB$ is special.  However,
it is usually {\em{not}} very special or strictly special when $\sB$ is so.  One strictifies the pseudofunctor to a functor at the price of losing strictness conditions on $\de$. 
When we apply the classifying space functor $\bB$ to a special $\sF$-category $\sB$, we obtain a special $\sF$-space; that is, the equivalences $\de$ become homotopy equivalences.  Therefore $\bB\sB$ is a special $\sF$-space in the standard sense: the space level maps $B\de$, often called Segal maps,  are homotopy equivalences.   Under $\bB\com \bf{St}$, the higher homotopies move from pseudocategorical input {\em{structure}} to standard homotopical output {\em{properties}}! 
\end{rem}

%\begin{defn}  The category $\FPs$ consists of pseudofunctors $\sF \rtarr \textbf{Cat}$ whose restriction to $\Pi$ are special $\Pi$-algebras, and pseudonatural transformations whose restriction to $\Pi$ are strict natural transformations. \end{defn}

The following two elementary results compare $\FPs$ and $\FPR$.   Their proofs are deferred to \autoref{FDetails}.  

\begin{prop} \label{ImageOfRPSAlg}
Every very special $\sF$-pseudoalgebra is naturally isomorphic to an $\sF$-pseudoalgebra whose restriction to $\PI$ is in the image of $\bR$.
%Therefore $\FPR$ is the sub $2$-category of $\FPs$ given by those $\sF$-pseudoalgebras whose restriction to $\Pi$ is in the image of $\bR$.
\end{prop}

\begin{prop}\label{FPsFPR}
$\FPs$ is equivalent as a 2-category to $\FPR$. 
\end{prop}

We briefly describe $\sE$ and $\sF$ combinatorially and then give a result that explains why we prefer to work in $\FPR$ rather than in $\FPs$. 
 
Intuitively, the morphisms of $\sF$ are composites of projections, maps in $\PI$, and wedges of canonical multiplication maps $\phi_n \colon \textbf{n} \rtarr \textbf{1}$. To make this  precise, we must first be precise about wedges.  We identify $\bm\vee\bn$ with $\mathbf{m+n}$, using the first $m$ and last $n$ integers in $\{1, 2, \cdots, m+n\}$.

\begin{defn}
 Given two maps $f \colon \textbf{m} \rtarr \textbf{n}$ and $g \colon \textbf{m}' \rtarr \textbf{n}'$ in $\sF$, we define their wedge $f \vee g \in \sF(\textbf{m+m}', \textbf{n+n}')$ to be the map that restricts to $\io_n f$ on the first $m$ (positive) values  and restricts to $\io_{n}'g$ on the last $m'$ values, where 
 %    {\textbf m} & {\textbf n} & {\textbf {n+n}'}
%    \arrow["f", from=1-1, to=1-2]
%    \end{tikzcd}\]
%    \[\begin{tikzcd}
%    {\textbf m'} & {\textbf n} & {\textbf {n+n}'}
%    \arrow["g", from=1-1, to=1-2]
%    \arrow["{in_{n'}}", from=1-2, to=1-3]
%    \end{tikzcd}\]
$\io_n $ and $\io_{n}'$ are the inclusions of the first $n$ and last $n'$ (positive) values in $\textbf{n+n}'$.
\end{defn}

\begin{defn} Let $\sA$ be a based category.  We define the product  \linebreak
$f\times g \colon \sA^{m+m'}\rtarr \sA^{n+n'}$ 
of functors $f\colon \sA^m\rtarr \sA^n$ and $g\colon \sA^{m'}\rtarr \sA^{n'}$ to be the cartesian product of the functors $f\pi_m$ and $g\pi_{m'}$, 
%    \[\begin{tikzcd}
%
%    {\sA^{m+m'}} & {\sA^m} & {\sA^n}
%    \arrow["{proj_m}", from=1-1, to=1-2]
%    \arrow["f", from=1-2, to=1-3]
%    \end{tikzcd}\]
%    and
%    \[\begin{tikzcd}
%    {\sA^{m+m'}} & {\sA^{m'}} & {\sA^{n'}}
%    \arrow["g", from=1-2, to=1-3]
%    \end{tikzcd}\]
where $\pi_m$ and $\pi_{m'}$ are the projections from  $\sA^{m+m'}$ to its first $m$ and  last $m'$ coordinates.
\end{defn}

Clearly these operations are both associative.

\begin{prop}\label{decompose}
 Every map in $\sE$ is a composite of a permutation and a wedge of canonical product maps.  Every map in $\sF$ is a composite of a projection and a map in $\sE$. 
\end{prop}

\begin{proof} Let $\xi\in\sE(\bm, \bn)$ and set $m_j=|\xi^{-1}(j)|$ for $1\leq j\leq n$.  Some $m_j$ may be $0$, but since we are in $\sE$,  $\sum_{1\leq j\leq n} m_j = m$.   The sets $\xi^{-1}(j)$ are ordered as subsets of $\bm$, and we define $\si\colon \bm\rtarr \bm$ to be the permutation that maps the ordered set $\{1, \cdots, m\}$ to its reordering as the block sum of the ordered sets $\xi^{-1}(j)$ for those $j$ such that $m_j> 0$. (There is no $j$th block when $m_j=0$.)    Then $\xi\si = \ph_{m_1}\vee \cdots \vee \ph_{m_n}$.  

For the second statement, any  $\ps \in \sF(\bm,\bn)$ is the composite of the projection $\pi\colon \bm\rtarr \bk$ that sends $\ps^{-1}(0)$ (and no other elements) to $0$ and a map  $\xi\in\sE$.
\end{proof}
%\begin{proof}
%Let $f \in \sF(\textbf{m},\textbf{n})$. Consider the set $f^{-1}(0)$, and take the projection $\pi \colon \textbf{m}\rtarr\textbf{k}$ sending exactly $f^{-1}(0)$ to 0. It is clear that we have a triangle 
%    \[\begin{tikzcd}
%    {\textbf m} & {\textbf k} \\
%    \arrow["\pi", from=1-1, to=1-2]
%    \arrow["f"', from=1-1, to=2-2]
%    \arrow["{f'}", from=1-2, to=2-2]
%    \end{tikzcd}\]
%    for some $f' \colon \textbf{k} \rtarr \textbf{n}$ with $f'^{-1}(0)=\{0\}$. Next, we consider the sets $f'^{-1}(j)$ for $j>0$. There is some permutation $\sigma \colon \textbf{k} \rtarr \textbf{k}$ which reorders $\textbf{k}$ as $f'^{-1}(1)\sqcup f'^{-1}(2)\sqcup \cdots \sqcup f'^{-1}(n)$. Then for $f_j = |f^{-1}(j)|=|f'^{-1}(j)|$ for $j>0$, we claim that $(\phi_{f_1} \oplus \phi_{f_2}\oplus\cdots\oplus \phi_{f_n})\circ \sigma = f'$. But this is automatic, since if $i \in f'^{-1}(j)$, then $f_1+\cdots+f_{j-1}<\sigma(i)\leq f_1+\cdots+f_{j-1}+f_j$, so the direct sum will map $\sigma(i)$ to $j$. Diagrammatically:
%    \[\begin{tikzcd}
 %   {\textbf m} & {\textbf k} & {\textbf k} \\
%    & {\textbf n}
%    \arrow["f"', from=1-1, to=2-2]
%    \arrow["\sigma", from=1-2, to=1-3]
%    \arrow["{f'}", from=1-2, to=2-2]
%    \arrow["{\phi_{f_1}\oplus \cdots\oplus \phi_{f_n}}", from=1-3, to=2-2]
%    \end{tikzcd}\]
%.  \end{proof}

The following observation gives one reason why we prefer to work in $\FPR$.

\begin{prop}\label{wedprod} Let $\sB\in \FPR$, so that $\sB$ restricts to $\bR\sA$ on $\PI$, where $\sA = \sB(1)$.  Then, up to natural isomorphism,
 $\sB$ takes wedges to products.
\end{prop} 
\begin{proof} For maps $f\colon \bm\rtarr \bn$ and $g \colon \bm' \rtarr \bn'$ in $\sF$ we must  show that $\sB(f\vee g)$ and $\sB(f)\times \sB(g)$ are naturally isomorphic functors $\sA^{m+m'} \rtarr \sA^{n+n'}$.  As a functor on $\PI$, $\sB$ takes the projection $\pi_n\colon \mathbf{n+n}' \rtarr \bn$ in $\PI$ to the projection $\pi_n\colon \sA^{n+n'}\rtarr \sA^{n}$ in $\mathbf{Cat}_*$ and similarly for $n'$.  Therefore, by the definition of $\vee$ and pseudofunctoriality, we have
$$ \pi_n(\sB(f) \times \sB(g)) = \sB(f) \pi_m = \sB(f) \sB(\pi_m) \iso \sB(f\pi_m)$$
and
$$ \pi_n\sB(f\vee g) = \sB(\pi_n)\sB(f\vee g) = \sB(\pi_n(f \vee g)) \iso \sB(f\pi_m).$$
Similarly,
$$\pi_{n'}(\sB(f)\times \sB(g)) \iso \sB(g\pi_{m'})\iso \pi_{n'}\sB(f\vee g).$$
 The product of these natural isomorphisms is the desired natural isomorphism 
 $$\sB(f) \times \sB(g) \iso \sB(f\vee g).$$
 \end{proof}

%    Write $\sC(1)=\sD$. Suppose we have two $\sF$-morphisms $f \colon \textbf{m} \rtarr \textbf{n}$ and $g \colon \textbf{m}' \rtarr \textbf{n}'$. These correspond to functors $\sC(f)\colon \sD^m\rtarr\sD^n$ and $\sC(g)\colon \sD^{m'}\rtarr\sD^{n'}$. We would like to show that the functors $\sC(f)\oplus \sC(g) \colon \sD^{m+m'}\rtarr\sD^{n+n'}$ and $\sC(f\oplus g) \colon \sD^{m+m'}\rtarr\sD^{n+n'}$ are naturally isomorphic. Since $\sC \in \FPR$, $\sC$ maps projections in $\sF$ to the analogous projections in $\textbf{Cat}_*$. So for $proj_n \colon \textbf{n+n}'\rtarr\textbf{n}$ the projection to the first $n$ coordinates and $proj_{n'} \colon \textbf{n+n}'\rtarr\textbf{n}'$ the projection to the last $n'$ coordinates, we have $\sC(proj_{n})=proj_n$ and $\sC(proj_{n'})=proj_{n'}$. So we can postcompose with these:
%    \begin{align*}
%        proj_n\circ(\sC(f)\oplus\sC(g)) &= \sC(f)\circ proj_m \\
%        &= \sC(f)\circ \sC(proj_m) \\
%        proj_n\circ\sC(f\oplus g) &= \sC(proj_n)\circ \sC(f\oplus g)\\
%        &\simeq \sC(proj_n \circ (f\oplus g)) \\
%        &= \sC(f \circ proj_m)
%    \end{align*}
%    Similarly:
 %   \[ proj_{n'}\circ(\sC(f)\oplus\sC(g)) \simeq \sC(g \circ proj_{m'}) \simeq proj_{n'}\circ\sC(f\oplus g) \]
 %   Since we have these natural isomorphisms, we can take their product to get a natural isomorphism 
 %   \[ \sC(f)\oplus\sC(g) \simeq \sC(f\oplus g) \]
%   \end{proof}

\section{Sketch proof of \autoref{thm2}}\label{Sec2}
\subsection{The $2$-functor $\bQ\colon \FPR \rtarr \PPs$}\label{SecQ}
%Since we are now thinking $2$-categorically, we change our point of view on the operad $\sP$.

%\begin{rem} A monoid can be considered as a category in two different ways.  We can consider its elements as objects and its product as composition, which was our implicit point of view when considering $\sM$ in \autoref{PPs}.  Alternatively, we can think of it as a category with a single object, with its elements being morphisms from that one object to itself.   We now think of $\sP$ $2$-categorically, so that each $\sP(n)$ has a single object $\bn$, with $\sI(\SI_n)$ as the category of morphisms from $\bn$ to $\bn$.   Thus we think of $\bn$ as a $0$-cell, the elements of $\SI_n$ as $1$-cells, and the (unique) morphisms between pairs of $1$-cells as $2$-cells. \end{rem}

Let $\sB\in \FPR$, so that the $\sF$-pseudoalgebra $\sB$ restricts to the $\PI$-category  $\bR \sA \colon \PI \rtarr \mathbf{Cat}$, where $\sA = \sB(1)$.   We have action functors 
\begin{equation}\label{Psi}
\Psi\colon \sF(\bm,\bn)\times \sA^m\rtarr \sA^n.
\end{equation} 
These give ``vertical" functoriality.   Pseudofunctoriality (\autoref{PsFun}) means intuitively that ``horizontal" functoriality is not strict.  Instead, we have suitably equivariant and coherent natural isomorphisms $\ph_1$ and $\ph$ in the diagrams

\begin{equation}\label{FPRdiag}
\begin{tikzcd}
{\sA} & {\sF(\mathbf{1},\mathbf{1})\times \sA} & 
    {\sF(\textbf n, \textbf k)\times \sF(\textbf m, \textbf n) \times \sA^m} & {\sF(\textbf n, \textbf k) \times \sA^n} \\
  & \sA&
    {\sF(\textbf m,  \textbf k) \times \sA^m} & \sA^k.
    \arrow["{\io}", from=1-1, to=1-2]
   \arrow[""{name=0, anchor=center, inner sep=0}, "\id"', from=1-1, to=2-2]
    \arrow["{\phi_1}"{pos=0.4}, shift right, shorten <=3pt, shorten >=3pt, Rightarrow, from=1-2, to=0]
    \arrow["{\PS}", from=1-2, to =2-2]
    \arrow[""{name=0, anchor=center, inner sep=0}, "{\Id \times\Psi}", from=1-3, to=1-4]
    \arrow["{\com\times \Id}"', from=1-3, to=2-3]
    \arrow["\Psi", from=1-4, to=2-4]
    \arrow[""{name=1, anchor=center, inner sep=0}, "\Psi"', from=2-3, to=2-4]
    \arrow["\phi", shorten <=9pt, shorten >=9pt,  Rightarrow, from=1-4, to=2-3]
    \end{tikzcd}
\end{equation}
Here $\io$ is the evident identification; \autoref{norm1} gives comments on the first diagram.

\begin{rem}\label{keyc} The central case is when $m=n$ and $k=1$, restricted from $\sF$ to $\sE$.  The domain then contains
$\sE(\bn,\mathbf{1})\times \LA(\bn,\bn)\times \sA^n$.  This case generates the others under wedges, and it focuses on the example in  \autoref{example}.  For $\si\in \SI_n$ and $\be$ an object or morphism of $\sA^n$,  $\ph$ gives an isomorphism from the permuted canonical product $\PS(\ph_n, \si \be)$ to the canonical product $\PS(\ph_n\si, \be) = \PS(\ph_n,\be)$. 
\end{rem}

On objects, we define $\bQ \sB = \sA$.   We must define action functors 
$$\tha = \tha_n \colon \sP(n) \times \sA^n \rtarr \sA.$$  
Notice that $\ph_n$ is the unique element of  $\sE(\bn,\mathbf{1})$, whereas $\sP(n) = \sE\SI_n$ has object set $\SI_n$, with the identity element $e_n$ as a privileged object.  Identifying $\ph_n$ and $e_n$, this gives an 
identification  of the object set $\SI_n$ of $\sP(n)$ with the set 
$$\{\ph_n\} \times \SI_n =\sE(\bn,\mathbf{1})\times \LA(\bn,\bn) \subset \sF(\bn,\mathbf{1})\times \sF(\bn,\bn). $$
We define $\tha$ on objects by restricting the composite
$$\xymatrix@1{ \sF(\bn,\mathbf{1}) \times \sF(\bn,\bn) \times \sA^n \ar[r]^-{\Id\times\PS} & \sF(\bn,\mathbf{1})\times \sA^n \ar[r]^-{\PS} &  \sA\\}$$
to $\sE(\bn,\mathbf{1})\times \LA(\bn,\bn)\times \sA^n$.  Here the first action map $\PS$ is the (left) action of $\SI_n$ on $\sA^n$.  With $\ph_n$ identified with $e_n$,  the second action map $\PS$ specifies $\tha$ on $e_n \times \be$, where $\be$ is an object or morphism of $\sA^n$.  We must define $\tha$ on the morphisms of $\sP(n)$.  Writing $u_{\si,\ta}$ for the unique morphism 
$\si\to \ta$ in $\sP(n)$, the (right) action of $\SI_n$ on $\sP(n)$ sends the pair $(u_{\rh,\si}, \ta)$ to $u_{\rh \ta, \si\ta}$.  The action restricts to an isomorphism
$$ \sP(n)(\SI_n,e_n)\times \SI_n  \rtarr  \mathbf{Mor}\sP(n);$$
its  inverse sends the unique morphism $u_{\ta,\si}$ to the pair $(u_{\ta\si^{-1},e_n}, \si)$. Under this isomorphism, composition of morphisms in $\sP(n)$ corresponds to the composition on $\sP(n)(\SI_n,e_n)\times \SI_n$ that sends  
$$ ((u_{\nu,e_n}, \ta), (u_{\mu,e_n},\si))  \ \ \text{to} \ \ (u_{\mu\nu,e_n}, \ta).$$
We define $\tha(\id_{e_n}, \be) = \PS(\ph_n, \be)$ for a morphism $\be\in \sA^n$ thus identifying the canonical $n$-fold products in $\sA$ and $\bQ\sA$.   Writing $\tilde{\si} = u_{\si,e_n}$, we define 
\begin{equation}\label{keyperm}
\tha(\tilde{\si},e_n, \be) = \PS(\ph_n,\si\be)
\end{equation}
as described in \autoref{keyc}.  The equivariance formula \autoref{PPsDefn}(i) requires
$\tha(\al\si,\be) = \tha(\al,\si\be)$ for a morphism $\al\in\sP(n)$, a permutation $\si\in \SI_n$ and a morphism $\be\in \sA^n$, thus giving
\begin{equation}\label{keyperm2}
\tha(\tilde{\si},\ta, \be) = \tha(\tilde{\si},e_n, \ta\be) = \PS(\ph_n,\si\ta\be)
\end{equation}
Extending $\phi$ in \autoref{keyc} by equivariance as in \autoref{PPsDefn}, we obtain the following result.
Here proj is short for projection. 
\begin{lem}\label{cando} There is a natural isomorphism $\ph_n$ in the diagram
\[\begin{tikzcd}
{\sP(n)\times \sA^n} &  \sA \\
{\sE(\bn,\mathbf{1}) \times \sA^n}  & 
\arrow["{\tha}", from=1-1, to =1-2]
\arrow["\textnormal{proj}"', from=1-1, to =2-1]
\arrow[""{name=0, anchor=center, inner sep=0}, "\PS"', from=2-1, to=1-2]
\arrow["{\phi_n}"{pos=0.4}, shorten <=3pt, shorten >=3pt, Rightarrow, from=1-1, to=0]
\end{tikzcd}
\]
\end{lem}

This lemma will be the key to verifying the first of the following two intuitive results.
We defer discussion of their proofs to \autoref{Qdetails}.

\begin{prop}\label{QProp1}
For each $\sB\in \FPR$,  $\bQ \sB$ is a $\sP$-pseudoalgebra.
\end{prop}

\begin{prop}\label{QProp2}
$\bQ$ extends to  a $2$-functor $\bQ\colon \FPR \rtarr \PPs$.
\end{prop}

\subsection{The $2$-functor $\bR\colon \PPs \rtarr \FPR$}\label{R}

Intuitively, the idea of $\bR$ is that when restricted to a $\sP$-pseudoalgebra $\sA$, the functor $\bR \sA\colon \PI \rtarr \bf{Cat}$ naturally extends to a pseudofunctor $\PS \colon \sF \rtarr \bf{Cat}$, which is obviously strictly special.   
%We think of $\bR\sA$ as sending the $0$-cell $\bn$ to $\sA^n$ and a $1$-cell $\bm \rtarr \bn$ to a map $\sA^m\rtarr \sA^n$.    

We focus on extending $\bR$ from a functor on $\PI$ to a functor on $\sE$.  Using that projections in $\PI$ map to projections in $\bR\sA$, it is not hard to show that we can then  extend the resulting functor on $\sE$ to a functor on $\sF$, using the factorization of maps in $\sF$ as composites of projections and maps in $\sE$ given in \autoref{decompose}.  We start by defining 
$$\PS(\ph_n, \be) = \tha(e_n,\be)$$
for an object or morphism $\be$ of $\sA^n$, thus identifying the canonical $n$-fold products of $\sA$ and $\bR \sA$. For a map  
$\xi\colon \bm \rtarr \mathbf{n}$ in $\sE$, we define $\PS(\xi)\colon \sA^m \rtarr \sA^n$ following the proof of \autoref{decompose}.  For $1\leq j\leq n$, set  $m_j = |\xi^{-1}(j)|$ and, for $\si\in \SI_n$, let $\si_m\in \SI_m$ be the block sum permutation defined in \autoref{decompose}.  By \autoref{wedprod}, $\PS$ must take wedges of maps in $\sE$ to products of their actions, hence we define $\PS(\xi)$ by the commutative diagram
%functoriality and use of projections, any  functor $\bF\colon sF \rtarr \bf{Cat}$ takes sums to products,                                                                              
$$ \begin{tikzcd}
        {\sA^m} & {\sA^m}\\
        & {\sA^n}
        \arrow["\sigma_m", from=1-1, to=1-2]
        \arrow["{\PS(\xi)}"', from=1-1, to=2-2]
        \arrow["{\PS(\ph_{m_1})\times\cdots\times \PS(\ph_{m_n})}", from=1-2, to=2-2]
    \end{tikzcd}
   $$
The map $\tha_n(\tilde{\si})$ then gives the arrow $\Rightarrow$ in the following coherence diagram 
\begin{equation}\label{FPRdiag2}
 \xymatrix{
\sA^n \ar[d]_{\ps_n} \ar[r]^-{\si}  \drtwocell<\omit>   & \sA^n \ar[d]^{\psi_n}\\
\sA \ar[r]_{=} & \sA \\}
\end{equation}
for the pseudoaction of $\sE$ that we are constructing.  As in \autoref{keyc}, this is the key case of the diagram \autoref{FPRdiag}.   The remaining cases follow by decomposing maps of $\sE$ as in \autoref{decompose} and using \autoref{wedprod} to construct $\ph$ in \autoref{FPRdiag} in general from the $\ph$ in \autoref{FPRdiag2}. We defer details of the verifications of the following two intuitive results to \autoref{Rdetails}. 

\begin{prop}\label{RProp1}
 For each $\sA\in \PPs$,  $\bR\sA$ is an $\sF$-pseudoalgebra.
\end{prop}

\begin{prop}\label{RProp2}
    $\bR$ extends to a $2$-functor $\bR\colon \PPs \rtarr \FPR$.
\end{prop}

\subsection{The proof of \autoref{thm2}}\label{Sketch2}

Together with \autoref{FPsFPR}, the following theorem gives \autoref{thm2}. 

\begin{thm}
$\bQ$ and $\bR$ are inverse natural isomorphisms of $2$-functors.
\end{thm}
\begin{proof}
The underlying category $\sA$ of objects of $\PPs$ and $\FPR$ is preserved by both $\bR$ and $\bQ$, and by construction the canonical products $\sA^n\rtarr \sA$ agree under both $\bQ$ and $\bR$.  That is, under the composite $\bQ\bR$ or $\bR\bQ$, we see the identity of the left or the right side of the equality
$$ \tha(e_n,\be) =  \PS (\ph_n,\be),$$ 
where $\be$ is an $n$-tuple of objects or morphisms of $\sA$. By equivariance identifications, these identities remain true with $e_n$ replaced by $e_n\si$ on the left and $\be$ replaced by $\si\be$ on the right.   For the morphisms of $\sP$ and the transformations $\ph$ of $\sF$-pseudoalgebras in the key case of \autoref{keyc}, we have the following schematic diagram.  Its two triangles and its upper trapezoid commute trivially.
\[\begin{tikzcd}
    {\sP(n)\times \sA^n} & && {\sP(n)\times\sA^n} \\
    & {\sA^n} & {\sA^n} \\
    & \sA & \sA
    \arrow["{\id \times\si}", from=1-1, to=1-4]
    \arrow["\theta"', from=1-1, to=3-2]
   % \arrow["{\sigma^{-1}\times\sigma}", from=1-2, to=1-4]
    \arrow["\theta", from=1-4, to=3-3]
    \arrow["{e_n\times\id}"', hook', from=2-2, to=1-1]
    \arrow[""{name=0, anchor=center, inner sep=0}, "\sigma", from=2-2, to=2-3]
    \arrow["{\phi_n}"', from=2-2, to=3-2]
    \arrow["{e_n\times \id}", hook, from=2-3, to=1-4]
    \arrow["{\phi_n}", from=2-3, to=3-3]
    \arrow[shorten <=6pt, shorten >=6pt, Rightarrow, from=2-3, to=3-2]
    \arrow["=", from=3-2, to=3-3]
    \end{tikzcd}\]
 Thinking of $\sA$ as in $\FPR$, $\ph_n$ denotes its canonical $n$-fold product, obtained by applying 
 $\PS$ to $\ph_n\times \sA^n$, and $\ph$ gives the arrow  $\Rightarrow$.  The map $e_n$ is the inclusion of the trivial category in $\sP$ with image $e_n$.   Restricting $\sP(n)$ to that image, 
 $\tha(\tilde{\si})$ gives an arrow $\Rightarrow$ in the outer trapezoid.  These two transformations agree under either  $\bQ$ or $\bR$.  Since all other transformations $\ph$ and all other actions of $\tha$ on morphisms of $\sP(n)$ are generated by these under equivariance  and wedges (in the case of $\FPR$), these compatibilities are enough to show that $\bQ\bR$ and $\bR\bQ$ are isomorphisms on objects of the respective $2$-categories.  The corresponding identifications for the $1$-cells and $2$-cells of our $2$-categories are similar direct comparisons and are left to the reader. \end{proof}
   
\section{Sketch proof of \autoref{thm1}}\label{Sec3}

\subsection{The $2$-functor $\bU\colon \PPs \rtarr \SM$}\label{U}

Given a $\sP$-pseudoalgebra $\sA$, we define $\bU \sA = \sA$.  Its product $\otimes$ is determined by $\tha_2$ as the composite
\[\otimes \colon \sA\times \sA\xrightarrow{\cong} \{e_2\}\times \sA\times \sA\hookrightarrow \sP(2)\times \sA\times \sA\rightarrow \sA.\]
Its unit is the base object $e$ in $\sA$, that is the image of $\sP(0)$. We define the $2$-functor $\bU \colon \PPs \rtarr \SM$
by  $\bU \sA=(\sA,\otimes,e)$ on objects.  We will show that this elaborates to a $2$-functor, but we must first show that $\bU \sA$ is indeed a symmetric monoidal category.  We defer 
 proofs of the following results to \autoref{Udetails}.

\begin{prop}  For $\sA\in \PPs$, $\bU \sA$ is a symmetric monoidal category. \label{UProp1}
\end{prop}

\begin{prop} $\bU$ extends to a $2$-functor  $\PPs \rtarr \SM$. \label{UProp2} 
\end{prop}

\subsection{The $2$-functor $\bV\colon \SM \rtarr \PPs$}\label{V}

Given a symmetric monoidal category $\sA = (\sA,\otimes, e)$ we define a $\sP$-pseudoalgebra $\bV\sA$ with underlying $\PI$-algebra $\bR\sA$.  For objects $a_i\in \sA$, we define a canonical $n$-fold product  functor
$$\tha(e_n; a_1, \cdots, a_n) \equiv  a_1\otimes \cdots \otimes a_n$$ 
inductively. We start the induction with the given product for $n=2$ and then set
$$\tha(e_n; a_1, \cdots, a_n) =  \tha(e_{n-1}, a_1, \cdots, a_{n-1}) \otimes a_n \equiv a_1 \otimes\cdots \otimes a_n.$$
Writing $\bar{a}=(a_1,\cdots,a_n)$, we define $\tha(\si; \bar{a}) = \tha(e_n; \si\bar{a})$ for $\si\in\SI_n$.  With the appropriate unit and associativity isomorphisms, we obtain the first of the following two results. Discussion of their proofs are deferred to \autoref{Vdetails}. 

\begin{prop}  For each $\sA\in \SM$, $\bV \sA$ is a $\sP$-pseudoalgebra. \label{VProp1}
\end{prop}

\begin{prop} $\bV$ extends to a $2$-functor  $\SM \rtarr \PPs$. \label{VProp2} 
\end{prop}

 \subsection{The proof of \autoref{thm1}}\label{Sketch1}
 
 The following is a reformulation of \autoref{thm1}.
 
 \begin{thm}
$\bU$ and $\bV$ are inverse natural isomorphisms of $2$-functors.
\end{thm}
\begin{proof}
    Let $(\sA, \otimes, e)$ be a symmetric monoidal category. According to our definitions, it is clear that the underlying collections of objects and morphisms in $\sA$ and $\bU\bV \sA$ are the same. Also, via inspection the product, associators, unit, and commutativity isomorphisms in $\sA$ and $\bU\bV \sA$ are the same. Therefore, $\sA$ and $\bU\bV \sA$ are the same as symmetric monoidal categories and $\bU\bV  = \mathbf{\Id}$.
    
    Thus it suffices to construct a natural isomorphism $\ze\colon  \bV\bU\rtarr  \mathbf{\Id}$.  Here by a natural isomorphism we mean a pseudoisomorphism, as in \autoref{pseudomap}.  Let $\sA$ be a $\sP$-pseudoalgebra. Again,  the underlying category of  $\bV\bU \sA$ is $\sA$, so our equivalence starts with this identification of underlying categories.  Starting with the identity for $k=2$, we shall inductively define natural transformations $\zeta_k$ for $k\geq 2$ that make the following diagram commute and satisfy the other properties required of a $\sP$-pseudomorphism, as defined in \autoref{pseudomap}. 

\[\begin{tikzcd}
	{\sP(k)\times \sA^k} & {\sP(k)\times (\bV\bU \sA)^k} \\
	\sA & {\bV\bU \sA}
	\arrow["{\theta_k}"', from=1-1, to=2-1]
	\arrow["{=}"', from=1-2, to=1-1]
	\arrow["{\theta_k}", from=1-2, to=2-2]
	\arrow[""{name=0, anchor=center, inner sep=0}, "{=}", from=2-2, to=2-1]
	\arrow["{\zeta_k}"', shorten <=6pt, shorten >=6pt, Rightarrow, from=1-2, to=2-1]
\end{tikzcd}\]

Since both maps $\tha_k$ are (strictly) equivariant, a simple diagram chase shows that it suffices to define $\ze_k$ on objects 
$e_k\times \be$ for $k$-tuples $\be$ in $\sA$. Here and below, we do not distinguish notationally between objects and morphism in $\sA$ or in $\bV\bU \sA$.   On the left, $\tha_k$ is then the given canonical $k$-fold product.  On the  right, $\tha_k$ is our inductively defined canonical product.  The definition of a $\sP$-pseudoalgebra, \autoref{PPsDefn}, requires a natural isomorphism between these, and this is obtained inductively via the diagram
\[\begin{tikzcd}
	{\sP(2)\times (\sP(k-1)\times \sA^{k-1}\times \sP(1)\times \sA)} & & {\sP(2)\times \sA^2} \\
	{\sP(k)\times \sA^k} & & \sA
	\arrow[""{name=0, anchor=center, inner sep=0}, "{\id\times (\theta_{k-1}\times \theta_2)}", from=1-1, to=1-3]
	\arrow["{\text{Shuffle}}"', from=1-1, to=2-1]
	\arrow["{\theta_2}", from=1-3, to=2-3]
	\arrow[""{name=1, anchor=center, inner sep=0}, "{\theta_k}"', from=2-1, to=2-3]
	\arrow["{\varphi}", shorten <=6pt, shorten >=6pt, Rightarrow, from=1-3, to=2-1]
\end{tikzcd}\]
Since $\tha_2(e, a_1, a_2) = a_1\otimes a_2$ on objects $a_i$ of $\sA$, this gives an isomorphism.

In particular, when the object in the top left corner is of the form
$$e_2\times (e_{k-1}\times (a_1,...,a_{k-1})\times e_1\times a_k),$$
$\ph$ gives a natural isomorphism
$$  \tha_{k-1}(e_{k-1};a_1, \cdots, a_{k-1})\otimes a_k \Rightarrow \tha_k(e_k;a_1,\cdots, a_k)  $$
By the induction hypothesis, we have an iterated analogous isomorphism from the inductively defined canonical product 
$a_1\otimes \cdots \otimes a_{k-1}$ to the given canonical product $\tha_{k-1}(e_{k-1};a_1, \cdots, a_{k-1})$.  Taking the product  with $a_k$, and composing, this gives the desired natural isomorphism from the iteratively defined $k$-fold caononical product in $\bV\bU \sA$ to the given canonical $k$-fold product in $\sA$.  As said, the $\ze_k$ extend from $e_k$ to all $\si\in \SI_k$ by equivariance. 

To have a pseudomorphism, we must have equalities of pasting diagrams, but this is what MacLane's coherence theorem \cite[Section XI.1]{Mac} tells us, as is perhaps most clearly seen in its 
operadic (or monadic) version \cite[Theorem 3.2]{Lack}.  The pasting diagrams 
required by \autoref{PPsDefn} factor into pasting diagrams of the $\varphi$'s, and these are equal due to the coherence theorem.   That is, the identity of $\sA$ extends to a pseudomorphism $\bV\bU \rtarr \Id$, necessarily an isomorphism.  Compatibility with the $1$-cells and $2$-cells of our $2$-categories are similar direct comparisons and are left to the reader.  The verification is just checking that the pasting diagrams in the definition of a $\sP$-pseudoalgebra are complete.
 \end{proof}

\section{Diagrammatic definitions and categorical combinatorics}\label{Sec4}
\subsection{$2$-Categories, $2$-functors, and pseudofunctors}\label{2Cats}
While $2$-categories and $2$-functors should be familiar at this late stage, we briefly recall what they are.  The prime example is the $2$-category $\mathbf{Cat}$  whose objects (or $0$-cells) are categories, whose morphisms (or $1$-cells) are functors, and whose $2$-morphisms (or $2$-cells) are natural transformations.    
\begin{defn} A $2$-category is a category enriched in the category $\mathbf{Cat}$ and a $2$-functor is an enriched functor between $2$-categories.
\end{defn}
Thus a $2$-category $\sC$ has a set of objects ($0$-cells) and for each pair $(x,y)$ of objects a category of morphisms $\sC(x,y)$ (whose objects are called morphisms or $1$-cells) and whose morphisms are called $2$-morphisms or $2$-cells), together with unit functors $\id_x$ from the trivial $1$-object category $\ast$ to each $\sC(x,x)$ and composition functors $\odot \colon \sC(y,z)\times \sC(x,y) \rtarr \sC(x,z)$ for each triple $(x,y,z)$  such that composition is (strictly) unital and associative.   Of course, the (horizontal) composition $\odot$ is not to be confused with the (vertical) composition $\com$ in each category $\sC(x,y)$.  A $2$-functor 
$\bF\colon \sC\rtarr \sD$ is a function on objects and functors $\bF\colon \sC(x,y)\rtarr \sD(\bF x, \bF y)$ for each pair of objects 
$(x,y)$ that are unital and preserve composition, meaning  that the following diagrams of functors commute.

\begin{equation}\label{2fun}
\xymatrix{
\ast\ar[r]^-{\id_x} \ar[dr]_{\id_{\bF x}} & \sC(x,x) \ar[d]^{\bF}  
& & \sC(y,z) \times \sC(x,y) \ar[d]_{\bF\times\bF} \ar[r]^-{\odot} & \sC(x,z) \ar[d]^{\bF}\\
& \sD(\bF x,\bF x) & & \sD(\bF y, \bF z) \times \sD(\bF x, \bF y) \ar[r]_-{\odot} & \sD(\bF x, \bF z)\\} 
\end{equation}

We recall the standard $2$-categorical generalization of functors and natural transformations.
For more details,and a full enumeration of coherence conditions, see \cite[Sections 4.1, 4.2]{JY}.

\begin{defn}\label{PsFun}
    A \textit{pseudofunctor} $\bF\colon \sC \rtarr \sD$, where $\sC$ and $\sD$ are categories, consists of 
   a function  $\bF \colon \textbf{Ob}(\sC) \rtarr \textbf{Ob}(\sD)$, functions 
   $\bF \colon \sC(x,y) \rtarr \sD(\bF x,\bF y)$ for objects $x$ and $y$ of $\sC$, and
    natural identity and composition isomorphisms in the following two diagrams.
    
\[\begin{tikzcd}
 {\ast} & {\sC(x,x)} &  &  {\sC(y,z)\times\sC(x,y)} & {\sC(x,z)} \\
& {\sD(\bF x, \bF x)} & &     {\sD(\bF y,\bF z)\times \sD(\bF x,\bF y)} & {\sD(\bF x,\bF z)}
    \arrow["\id_x", from=1-1, to=1-2]
    \arrow["\id_{\bF x}"', from=1-1, to= 2-2]
    \arrow["\bF", from=1-2, to=2-2]
%    \arrow["{\phi}"{pos=0.4}, shift right, Rightarrow, from=1-2, to=0]
    \arrow["\circ", from=1-4, to=1-5]
    \arrow["{\bF\times \bF}"', from=1-4, to=2-4]
    \arrow["\bF", from=1-5, to=2-5]
    \arrow["\circ"', from=2-4, to=2-5]
    \arrow[shorten <=6pt, shorten >=6pt, Rightarrow, from=1-4, to=2-5]
    \end{tikzcd}\]
Similarly, for \textit{pseudotransformations},  instead of requiring that the usual square of a natural transformation commutes, we only require that it commutes up to coherent natural isomorphism.
\end{defn}

\begin{rem}[Normality]\label{norm1}  Throughout, there is a distinction to be made between unit or identity {\em{objects}} and identity {\em{operations}}.  We can ask for either to be either strict or up to isomorphism.  Symmetric monoidal categories have unit objects up to isomorphism, but there is no mention of an identity operation in the definition: it is tacitly just assumed to be the identity.  We say that a pseudofunctor is {\em{normal}} if the identity natural transformation is the identity, so that 
$\bF  \id_x = \id_{\bF x}$ for all $x$.  In the case of pseudofunctors $\sF \rtarr \mathbf{Cat}$, this is equivalent to saying that $\ph_1$ is the identity isomorphism in the first diagram of \autoref{FPRdiag}.
An analogous normal variant of $\sP$-pseudoalgebras is defined in \autoref{norm2}. 
Operadically, \autoref{eunit} shows that the unit object and unit operation are very closely related.
\end{rem}

\begin{rem}\label{mod} Since the $0$-cells and $1$-cells of $\FPR$ are suitable pseudofunctors and pseudotransformations, the $2$-cells might logically  be called pseudomodifications, but we shall call them $\sF$-$2$-cells.  In general, we view pseudomorphisms $\sA\colon \sC \rtarr \mathbf{Cat}$ in adjoint form,  so that they are given by functors $\sC(c,d) \times \sA(c) \rtarr \sA(d)$ for objects $c$ and $d$ of $\sC$.  In this form, a modification between pseudofunctors $\sA \rtarr sB$ is given by  pseudomorphisms  $\la_{\sA}\colon \sC(c,d)\times \sA(c) \rtarr \sA(d)$ and $\la_{\sB}$ that give equalities of suitable pasting diagrams.  We shall make this more explicit in our context in \autoref{2cell2}.
\end{rem}

\subsection{Details about $\SM$}\label{SMdetails}

\begin{defn}[Symmetric Monoidal Category]\label{SMObj}
A symmetric monoidal category $\sA = (\sA, \otimes, e)$ consists of a category $\sA$, a product functor $\otimes$, and a unit object $e$, together with natural isomorphisms
$$ \al = \al_{x,y,z}: (x\otimes y)\otimes z\rightarrow x\otimes (y\otimes z),$$
$$  \la = \lambda_x: e\otimes x\rightarrow x \ \ \text{and} \ \  \rh = \rho_x: x\otimes e\rightarrow x,$$
and
$$\be =\be_{x,y}: x\otimes y\rightarrow y\otimes x$$
\
such  that the following diagrams commute

\begin{enumerate}[(i)]
    \item \textbf{(Pentagon)}
\[\begin{tikzcd}
	& {(w\otimes x)\otimes (y\otimes z)} \\
	{((w\otimes x)\otimes y)\otimes z } && {w\otimes (x\otimes (y\otimes z))} \\
	\\
	{(w\otimes (x\otimes y))\otimes z} && {w\otimes ((x\otimes y)\otimes z)}
	\arrow["{\al_{w\otimes x,y,z}}"', from=1-2, to=2-1]
	\arrow["{\al_{w,x,y\otimes z}}", from=1-2, to=2-3]
	\arrow["{\al_{w,x,y}\otimes 1_z}"', from=2-1, to=4-1]
	\arrow["{\al_{w, x\otimes y, z}}"', from=4-1, to=4-3]
	\arrow["{1_w\otimes \al_{x,y,z}}"', from=4-3, to=2-3]
\end{tikzcd}\]
\item \textbf{(Triangle)} 
\[\begin{tikzcd}
	{(x\otimes e)\otimes y} && {x\otimes (e\otimes y)} \\
	& {x\otimes y}
	\arrow["{\al_{x,e,y}}", from=1-1, to=1-3]
	\arrow["{\rho_x\otimes 1_y}"', from=1-1, to=2-2]
	\arrow["{1_x\otimes \lambda_y}", from=1-3, to=2-2]
\end{tikzcd}\]
\item \textbf{(Hexagon)} 
\[\begin{tikzcd}
	{(x\otimes y)\otimes z} & {x\otimes (y\otimes z)} & {(y\otimes z)\otimes x} \\
	{(y\otimes x)\otimes z} & {y\otimes (x\otimes z)} & {y\otimes (z\otimes x)}
	\arrow["{a_{x,y,z}}", from=1-1, to=1-2]
	\arrow["{\beta_{x,y}\otimes 1_z}"', from=1-1, to=2-1]
	\arrow["{\beta_{x, y\otimes z}}", from=1-2, to=1-3]
	\arrow["{\al_{y,z,x}}", from=1-3, to=2-3]
	\arrow["{\al_{y,x,z}}"', from=2-1, to=2-2]
	\arrow["{1_y\otimes \be_{z,x}}"', from=2-2, to=2-3]
\end{tikzcd}\]
\item \textbf{(Symmetry)}   \xymatrix{
x\otimes e \ar[r]^-{\be_{x,e}} \ar[dr]_{\rh} & e\otimes x \ar[d]^{\la}  
&  x\otimes y \ar[r]^-{\be_{x,y}}  \ar[dr]_{=} & y\otimes x \ar[d]^{\be_{y,x}}\\
& x   & & x\otimes y \\}
\end{enumerate}
\end{defn}

Up to notation, the following theorem is quoted from \cite{Mac}.  Its intuition is that ``all diagrams commute'' in a symmetric monoidal category. 

\begin{thm}[MacLane Coherence Theorem]\label{SMCoh}
    Let $\sA$ be a symmetric monoidal category. Given a permutation $\sigma$ on $n$ letters, define the corresponding functor $\sigma_{\sA}\colon \sA^n\rightarrow\sA$ on a tuple $(a_1,...,a _n)$ by tensor products of the tuple after permutation by $\sigma$. Then there exists a function that assigns to every pair of functors $\sigma_\sA , \tau_\sA : \sA ^n\rightarrow\sA $ a unique natural isomorphism
\[
\mathrm{can}_{\sA }(\sigma, \tau) : \sigma_{\sA } \longrightarrow \tau_{\sA },
\]
called the \emph{canonical map} from $\sigma_\sA $ to $\tau_\sA $, such that
\begin{enumerate}[(i)]
    \item the identity map of $\sA $ and the associativity and symmetry isomorphisms $\alpha, \be$ are canonical;
    \item composition and tensor product preserve canonical maps.
\end{enumerate}
\end{thm}

%\begin{comment}
%\begin{thm}[Mac Lane Coherence Theorem]
%    In every symmetric monoidal category $\sA $, there exists a function that assigns to each pair $(v\sigma, w\tau)$ of permuted words of the same length $n$ a unique natural isomorphism
%\[\mathrm{can}_{\sA }(v\sigma, w\tau) : (v\sigma)_{\sA } \longrightarrow (w\tau)_{\sA } : \sA ^n \rtarr \sA , \]
%called the \emph{canonical map} from $v\sigma$ to $w\tau$. This assignment is made so that:
%\begin{enumerate}[(i)]
%    \item the identity map of $\sA $ and all instances of the associativity $\alpha$ and symmetry $\gamma$ isomorphisms are canonical;
%    \item the composition of canonical maps is canonical; and
%    \item the tensor product $\otimes$ of canonical maps is canonical.\end{enumerate}
%\end{thm} \end{comment}

%In the paper, we often refer to this theorem as ``the coherence theorem''.

\begin{defn}[Symmetric Monoidal Functor]\label{SMMap}
 Let $\sA$ and $\sB$ be symmetric monoidal categories. A symmetric monoidal functor between them is a functor $\bF\colon \sA\rtarr \sB$ together with coherence isomorphisms 
 $$\epz: e_{\sB} \rtarr \bF(e_{\sA}) \ \ \text \ \ \mu = \mu_{a_1,a_2}: \bF(a_1)\otimes \bF(a_2)\rtarr \bF(a_1\otimes a_2)$$
  for all $a_1,a_2\in \sA$  such that the following diagrams commute:
 \begin{enumerate}[(i)]
        \item \textbf{(Associativity)}
\[\begin{tikzcd}
	{(\bF(a_1) \otimes \bF(a_2)) \otimes \bF(a_3)} & {\bF(a_1) \otimes (\bF(a_2) \otimes \bF(a_3))} \\
	{\bF(a_1 \otimes a_2) \otimes \bF(a_3)} & {\bF(a_1) \otimes \bF(a_2 \otimes a_3)} \\
	{\bF((a_1 \otimes a_2) \otimes a_3)} & {\bF(a_1 \otimes (a_2 \otimes a_3))}
	\arrow["{\alpha^{\sB }}", from=1-1, to=1-2]
	\arrow["{\mu_{a_1,a_2} \otimes \mathrm{id}}"', from=1-1, to=2-1]
	\arrow["{\mathrm{id} \otimes \mu_{a_2,a_3}}", from=1-2, to=2-2]
	\arrow["{\mu_{a_1 \otimes a_2, a_3}}"', from=2-1, to=3-1]
	\arrow["{\mu_{a_1, a_2 \otimes a_3}}", from=2-2, to=3-2]
	\arrow["{\bF(\alpha^{\sA})}"', from=3-1, to=3-2]
\end{tikzcd}\]
\item \textbf{(Unitality)}   
\[\begin{tikzcd}
{\bF(a_1) \otimes e_{\sB }} & {\bF(a_1)} & {e_{\sB } \otimes \bF(a_1)} & {F(a_1)}\\
{\bF(a_1) \otimes \bF(e_{\sA})} & {\bF(a_1 \otimes e_{\sA})} & 	{\bF(e_{\sA}) \otimes \bF(a_1)} & {\bF(e_{\sA} \otimes a_1)}
	\arrow["{\rh^{\sB}}", from=1-1, to=1-2]
	\arrow["{\mathrm{id} \otimes \epz}"', from=1-1, to=2-1]
	\arrow["{\mu_{a_1, e_{\sA}}}"', from=2-1, to=2-2]
	\arrow["{\bF(\rh^{\sA})}"', from=2-2, to=1-2]
	\arrow["{\la^{\sB}}", from=1-3, to=1-4]
	\arrow["{\epz \otimes \mathrm{id}}"', from=1-3, to=2-3]
	\arrow["{\mu_{e_{\sA}, a_1}}"', from=2-3, to=2-4]
	\arrow["{\bF(\la^{\sA})}"', from=2-4, to=1-4]
\end{tikzcd}\]

\item \textbf{(Symmetry)}. 
\[\begin{tikzcd}
	{\bF(a_1) \otimes \bF(a_2)} & {\bF(a_2) \otimes \bF(a_1)} \\
	{\bF(a_1 \otimes a_2)} & {\bF(a_2 \otimes a_1)}
	\arrow["{\be^{\sB}}", from=1-1, to=1-2]
	\arrow["{\mu_{a_1, a_2}}"', from=1-1, to=2-1]
	\arrow["{\mu_{a_2, a_1}}", from=1-2, to=2-2]
	\arrow["{\bF(\be^{\sA})}"', from=2-1, to=2-2]
\end{tikzcd}\]    \end{enumerate}
\end{defn}

\begin{rem}
The  previous definition coincides with what is often referred to as a \textit{strong} symmetric monoidal functor, where being strong requires that the coherence maps be invertible. 
\end{rem}

\begin{defn}[Natural Transformations]\label{SMMap2}
Let $\sA$ and $\sB$ be symmetric monoidal categories with symmetric monoidal functors $\bF, \bG\colon \sA \rtarr \sB$. A natural transformation $\xi\colon \bF \rtarr \bG$ is symmetric monoidal if the following diagrams commute.
$$\xymatrix{
e_{\sB} \ar[r]^-{\epz^{\bF}} \ar[dr]_{e^{\bG}} & \bF(e_{\sA}) \ar[d]^{\xi}  
& \bF(a_1) \otimes \bF(a_2) \ar[r]^-{\mu^{\bF}} \ar[d]_{\xi\otimes \xi} & \bF(a_1\otimes a_2) \ar[d]^{\xi}\\
& \bG(e_{\sA}) 
& \bG(a_1) \otimes \bG(a_2) \ar[r]_-{\mu^{\bG}}  & \bG(a_1\otimes a_2)\\}
$$
\end{defn}

We denote by $\SM$ the $2$-category of symmetric monoidal categories, symmetric monoidal functors, and their natural transformations. 

%We will work with the category of $\sP$-pseudoalgebras, $\sP\textbf{-PsAlg}$, with morphisms being $\sP$-pseudomorphisms, and the category of symmetric monoidal categories, $\textbf{SymmMon}$, with morphisms being symmetric monoidal functors.

\subsection{Details about $\PPs$}\label{PPsdetails}

The objects of the category $\sP(j)$ are the permutations in $\SI_j$.  Let $\sigma\in \sP(k)$, $\tau_i\in \sP(j_i)$ and $j_+ = \sum^k_{i=1} j_i$. Then the operadic structure map 
$$\gamma: \sP(k)\times\sP(j_1)\times...\times\sP(j_k)\rightarrow\sP(j_+)$$ 
is defined on objects by
$$\gamma(\sigma; \tau_1,...,\tau_n) = (\times^k_{i=1} \tau_{\sigma^{-1}(i)})\circ \sigma(j_1,...,j_k),$$ 
where $\sigma(j_1,...,j_k)$ is the permutation corresponding to $\sigma$ that permutes the $k$ blocks of letters in $j_+$, with the $i$-th block containing $j_i$ letters. Both $\sP(0)$ and $\sP(1)$ are copies of the trivial category.  We think of $\sP(0)$ as giving a base object $\ast$ and $\sP(1)$ as giving an identity operation $\io$.  We think of $\io$ as the identification 
$\ast \iso \sP(1)$.

We mostly follow Corner and Gurski \cite{CG} in defining $\sP$-pseudoalgebras.  We are motivated by comparison with the much older notion of a pseudoalgebra over a $2$-monad, as in \cite{Power}, for example.  The only differing details concern  the identity operation, which is more enmeshed with composition operadically than monadically. 

\begin{defn}[$\sP$-pseudoalgebra]\label{PPsDefn}
A \emph{$\sP$-pseudoalgebra} is a category $\sA$  and operad action functors 
    \[
      \theta_j:\ \sP(j)\times \sA^j \longrightarrow\ \sA
    \]
for each $j\in\mathbb N$, together with an invertible natural transformation $\ph_1$ in the diagram
\begin{equation}\label{phione}
\begin{tikzcd}
	{\sA} & {\sP(1)\times \sA} \\
	& \sA \\
	\arrow["{\io}", from=1-1, to=1-2]
	\arrow[""{name=0, anchor=center, inner sep=0}, "\id"', from=1-1, to=2-2]
	\arrow["{\theta_1}", from=1-2, to=2-2]
	\arrow["{\phi_1}"{pos=0.4}, shift right, shorten <=3pt, shorten >=3pt, Rightarrow, from=1-2, to=0]
\end{tikzcd}
\end{equation}
in which  $\io$ is the identification $\sA \iso \ast\times \sA \iso \sP(1)\times \sA$
and an invertible  natural transformation $\phi = \phi(k;j_1,...,j_k)$ for each diagram 
\begin{equation}\label{phikey}
\begin{tikzcd}
	{\sP(k)\times (\prod_i\sP(j_i)\times \sA^{j_i})} & {\sP(k)\times \sA^k} \\
	& \sA \\
	{\sP(k)\times \prod_i \sP(j_i)\times \sA^{j_+}} & {\sP(j_+)\times \sA^{j_+}}
	\arrow["\id\times \prod_i\tha_{j_i}", from=1-1, to=1-2]
	\arrow["{\text{shuffle}}"', from=1-1, to=3-1]
	\arrow["\tha_k", from=1-2, to=2-2]
	\arrow["{\phi}", shorten <=9pt, shorten >=9pt, Rightarrow, from=1-2, to=3-1]
	\arrow["\ga\times \id"', from=3-1, to=3-2]
	\arrow["\tha_{j_+}"', from=3-2, to=2-2]
\end{tikzcd}
\end{equation}
Since $\sA^{0}$ is the trivial category, $\tha_{0}$ gives $\sA$ a base object $e$, which will be the unit object.  Both the $\tha_j$ and the $\phi$ are required to be (strictly) equivariant. 
\begin{enumerate}[(i)]
\item {\textbf{(Equivariance of $\tha_j$)}} The following diagram commutes for $\rho\in \Sigma_j$:
\[\begin{tikzcd}
	{\sP(j)\times \sA^j} & {\sP(j)\times \sA^j} \\
	{\sP(j)\times \sA^j} & \sA
	\arrow["{1\times \rho}", from=1-1, to=1-2]
	\arrow["{\rho \times 1}"', from=1-1, to=2-1]
	\arrow["{\theta_j}", from=1-2, to=2-2]
	\arrow["{\theta_j}"', from=2-1, to=2-2]
\end{tikzcd}\label{diag:example}\]
\item \textbf{(Equivariance of $\phi$)} The following equalities hold.  For  $\sigma\in \SI_k$,  
\begin{center}
    $\phi(k;j_1,...,j_k) = \phi(k; j_{\si(1)},...,j_{\si(k)})\circ (\sigma\times \sigma^{-1})$
\end{center}
For $\tau_i\in \SI_{j_i}$, 
\begin{center}
    $\phi(k;j_1,...,j_k) = \phi(k; j_1,...,j_k)\circ (\id \times \prod_i (\tau_i\times \tau_i^{-1}))$.
\end{center}

The $\tha_j$ and $\phi$ must satisfy the following coherence properties.  

\item \textbf{(Operadic Composition)}.  The following pasting diagrams are equal, where \linebreak
$j_+ = \sum j_r$, $l_+ = \sum_{r,s} l_{r,s}$, and $l_r = \sum_s l_{r,s}$ with $1\leq r\leq k$ and $1\leq s\leq j_r$.

{\tiny{
\[\begin{tikzcd}
	{\sP(k)\times \prod_r(\sP(j_r)\times \sP(l_{r,s})\times \sA^{l_{r,s}})} & &{\sP(k)\times \prod_r(\sP(j_r)\times \sA^{j_r})} \\
	& & & {\sP(k)\times \sA^k} \\
	{\sP(j_+)\times \prod_{r,s}(\sP(l_{r,s})\times \sA^{l_{r,s}})} & & {\sP(j_+)\times \sA^{j_+}} \\
	& {\sP(l_+)\times \sA^{l_+}} & & \sA \\
	&& {}
	\arrow["{\id\times(\prod \id\times \theta_{l_{r,s}})}", from=1-1, to=1-3]
	\arrow["\ga"', from=1-1, to=3-1]
	\arrow["\id\times(\prod_r \tha_{j_r})", from=1-3, to=2-4]
	\arrow["\ga", from=1-3, to=3-3]
	\arrow["\phi", shorten <=6pt, shorten >=6pt, Rightarrow, from=2-4, to=3-3]
	\arrow["{\theta_k}", from=2-4, to=4-4]
	\arrow["{\id\times\prod\theta_{l_{r,s}}}", from=3-1, to=3-3]
	\arrow["\ga"', from=3-1, to=4-2]
	\arrow["\phi", shorten <=6pt, shorten >=6pt, Rightarrow, from=3-3, to=4-2]
	\arrow["{\theta_{j_+}}", from=3-3, to=4-4]
	\arrow["{\theta_{l_+}}"', from=4-2, to=4-4]
\end{tikzcd}\]
}}
equals
{\tiny{ 
\[\begin{tikzcd}[cramped]
	{\sP(k)\times \prod_r(\sP(j_r)\times \sP(l_{r,s})\times \sA^{l_{r,s}})} && {\sP(k)\times \prod_r(\sP(j_r)\times \sA^{j_r})} \\
	& {\sP(k)\times \prod(\sP(l_r)\times \sA^{l_r})} && {\sP(k)\times \sA^k} \\
	{\sP(j_+)\times \prod_{r,s}(\sP(l_{r,s})\times \sA^{l_{r,s}})} \\
	& {\sP(l_+)\times \sA^{l_+}} && \sA
	\arrow["{\id\times(\prod \id\times \theta_{l_{r,s}})}", from=1-1, to=1-3]
	\arrow["{\id\times \ga}"', from=1-1, to=2-2]
	\arrow["\ga"', from=1-1, to=3-1]
	\arrow["{\id\times \phi}", shorten <=6pt, shorten >=6pt, Rightarrow, from=1-3, to=2-2]
	\arrow["\id\times(\prod_r \tha_{j_r})", from=1-3, to=2-4]
	\arrow["{\id\times\prod \theta_{l_r}}"', from=2-2, to=2-4]
	\arrow["\ga", from=2-2, to=4-2]
	\arrow["\phi", shorten <=9pt, shorten >=9pt, Rightarrow, from=2-4, to=4-2]
	\arrow["{\theta_k}", from=2-4, to=4-4]
	\arrow["\ga"', from=3-1, to=4-2]
	\arrow["{\theta_{l_+}}"', from=4-2, to=4-4]
\end{tikzcd}\]
}}
Here the maps labelled $\ga$ are composites of shuffle isomorphisms and maps $\ga\times \id$.  The left rectangle in the top diagram commutes by naturality and the left rectangle in the bottom diagram commutes by the associativity in the definition of an operad.  

\item \textbf{(Operadic Identity and Composition)}   The following two pairs of pasting diagrams are equal.
{\footnotesize{ 
\[\begin{tikzcd}[cramped]
{\sP(k) \times\sA^k} & {\sP(k)\times (\sP(1) \times \sA)^k} & {\sP(k)\times \sA^k}  \\
& {\sP(k)\times \sP(1)^k \times \sA^k} &   \\
& {\sP(k) \times\sA^k} & {\sA} \\
\arrow["{\id\times\io^k}", from=1-1, to=1-2]
\arrow["{\id\times \tha_1^k}", from=1-2, to=1-3]
\arrow["{shuffle}"', from=1-2, to=2-2]
\arrow["{\tha_k}", from=1-3, to=3-3]
\arrow["{\ga\times\id}"', from=2-2, to=3-2]
\arrow["{\tha_k}"', from=3-2, to=3-3]
\arrow["{\ph(k;1^k)}", shorten <=6pt, shorten >=6pt, Rightarrow, from=1-3, to=3-2]
\end{tikzcd} \]
}}
equals
{\footnotesize{
\[\begin{tikzcd}[cramped]
 &  {\sP(k)\times (\sP(1) \times \sA)^k} &  &\\
 {\sP(k) \times\sA^k} & & {\sP(k)\times \sA^k}  & {\sA}\\
 {\sP(k)\times (\sP(1) \times \sA)^k}  & {\sP(k)\times \sP(1)^k \times \sA^k}  &   \\
 \arrow["{\id}"', from=2-1, to=2-3]
 \arrow[""{name=0, anchor=center, inner sep=0}, "\id"', from=2-1, to=2-3]
 \arrow["{\ph_1}"{pos=0.3}, shift left, shorten <=6pt, shorten >=6pt, Rightarrow, from=1-2, to=0]
 \arrow["{\id\times \io^k}", from=2-1, to=1-2]
 \arrow["{\id\times \tha_1^k}", from=1-2, to=2-3]
 \arrow["{\tha_k}", from=2-3, to=2-4]
  \arrow["{\id\times\io^k}"', from=2-1, to=3-1]
 \arrow["{shuffle}"', from=3-1, to=3-2]
 \arrow["{\id\times \io}"', from=2-1, to=3-2]
 \arrow["{\ga\times\id}"', from=3-2, to=2-3]
 \end{tikzcd} \]
 }}
 Here the lower left triangle commutes trivially and the lower middle triangle commutes since $\ga\io = \id$ in the operad $\sP$.  The left triangle in the next diagram commutes similarly; the top rectangle in the diagram that follows commutes trivially.
 {\footnotesize{ 
 \[\begin{tikzcd}[cramped]
 {\sP(k) \times\sA^k} & {\sP(1) \times \sP(k) \times\sA^k} & {\sP(1) \times\sA}  \\
 & {\sP(k) \times\sA^k} & {\sA}  \\
 \arrow["{\io}", from=1-1, to=1-2]
 \arrow["{\id\times \tha_k}", from=1-2, to=1-3]
 \arrow["{\id}"', from=1-1, to=2-2]
 \arrow["{\ga\times\id}"', from=1-2, to=2-2]
 \arrow["{\tha_1}", from=1-3, to=2-3]
 \arrow["{\ph(1;k)}", shorten <=6pt, shorten >=6pt, Rightarrow, from=1-3, to=2-2]
 \arrow["{\tha_k}"', from=2-2, to=2-3]
 \end{tikzcd} \]
 }}
 equals
 {\footnotesize{
 \[\begin{tikzcd}[cramped]
  {\sP(k) \times\sA^k}  &  {\sP(1)\times \sP(k) \times\sA^k}  \\
  {\sA} &  {\sP(1) \times\sA} \\
  & {\sA} \\ 
 \arrow["{\io}", from=1-1, to=1-2]
 \arrow["{\tha_k}"', from=1-1, to=2-1]
 \arrow["{\id\times\tha_k}", from=1-2, to=2-2]
 \arrow["{\io}", from=2-1, to=2-2]
 \arrow["{\tha_1}", from=2-2, to=3-2]
 \arrow[""{name=0, anchor=center, inner sep=0}, "\id"', from=2-1, to=3-2]
 \arrow["{\phi_1}"{pos=0.4}, shorten <=6pt, shorten >=6pt, Rightarrow, from=2-2, to=0]
  \end{tikzcd} \] 
 }}
 \end{enumerate}
\end{defn}

\begin{rem}\label{eunit} The case $(2;0,1)$ and $(2;1,0)$ of \autoref{phikey} relate the unit object $e$ with the unit operation given by $e_1\in\sP(1)$.  Focusing on $e_2\in \sP(2)$ and writing 
$\tha_2(e_2, a, b) = a\otimes b$, the $\ph$ in those two instances of \autoref{phikey} give left and right unit object  isomorphisms
$$  e\otimes a \iso \tha_1(e_1; a) \iso a\otimes e $$
\end{rem}

\begin{rem}\label{norm2} In \cite{GMMO1}, we described a $\sP$-pseudoalgebra $\sA$ as {\em{normal}} if $\ph_1$ is the identity, so that the operation induced by $e_1\in \sP(1)$ is the identity.  This has the effect of strictifying all of the  operatic identity and composition diagrams.   We defined an analogous notion of a normal $\sF$-pseudoalgebra in \autoref{norm1}.  \autoref{thm2} has two correct interpretations, one restricted to normal  $\sP$ and $\sF$ pseudoalgebras, the other not.  \autoref{thm1} is more natural interpreted in terms of normal $\sP$-pseudoalgebras, since symmetric monoidal categories are intrinsically normal, as observed in \autoref{norm1}.
\end{rem}

\begin{defn}[$\sP$-pseudomorphism]\label{pseudomap}
    A $\sP$-pseudomorphism $(\bF,\zeta): (\sA, \theta, \phi)\rightarrow (\sB, \vartheta, \varphi)$ consists of a functor 
    $\bF\colon \sA\rtarr \sB$ such that for each $j$ the following diagram commutes up to a natural isomorphism $\zeta_j$.
\begin{equation}\label{Fzeta}
\begin{tikzcd}[cramped]
	{\sP(j)\times \sA^j} & {\sP(j)\times\sB^j} \\
	\sA & {\sB}
	\arrow["{\id\times \bF^j}", from=1-1, to=1-2]
	\arrow["{\theta_j}"', from=1-1, to=2-1]
	\arrow["{\zeta_j}", shorten <=6pt, shorten >=6pt, Rightarrow, from=1-2, to=2-1]
	\arrow["{\vartheta_j}", from=1-2, to=2-2]
	\arrow["\bF"', from=2-1, to=2-2]
\end{tikzcd}
\end{equation}
The pair $(\bF,\ze)$ must satisfy the following properties.
\begin{enumerate}[(i)]
\item \textbf{Equivariant} $\zeta_j = \zeta_j\circ (\sigma\times \sigma^{-1})$ for all $\sigma\in \Sigma_j$.
\item \textbf{Preserve Identity} The following two  pasting diagrams are equal: 
{\footnotesize{ 
\[\begin{tikzcd}[cramped]
	\sA & \sB \\
	%{\ast\times \sA} & {\ast\times \sB} \\
	{\sP(1)\times \sA} & {\sP(1)\times \sB} \\
	& \sA && \sB
	\arrow["\bF", from=1-1, to=1-2]
	\arrow["\io"', from=1-1, to=2-1]
	\arrow["\io"', from=1-2, to=2-2]
	\arrow[""{name=0, anchor=center, inner sep=0}, "\id", from=1-2, to=3-4]
	%\arrow["{\id \times \bF}", from=2-1, to=2-2]
	%\arrow["{\io\times \id}"', from=2-1, to=3-1]
	%\arrow["{\io\times \id}"', from=2-2, to=3-2]
	\arrow["{\id \times \bF}", from=2-1, to=2-2]
	\arrow["{\theta_1}"', from=2-1, to=3-2]
	\arrow["{\zeta_1^{-1}}", shorten <=2pt, shorten >=2pt, Rightarrow, from=3-2, to=2-2]
	\arrow["{\vartheta_1}"', from=2-2, to=3-4]
	\arrow["\bF"', from=3-2, to=3-4]
	\arrow["\varphi_1"', shorten <=1pt, shorten >=1pt, Rightarrow, from=2-2, to=0]
\end{tikzcd} \]

equals

\[ \begin{tikzcd}
	\sA & \sB \\
	 & \\
	 {\sP(1)\times \sA} & {\sA } && {\sB} 
	\arrow["\bF", from=1-1, to=1-2]
	\arrow["\io"', from=1-1, to=3-1]
	\arrow[""{name=0, anchor=center, inner sep=0}, "\id", from=1-1, to=3-2]
	\arrow["\id", from=1-2, to=3-4]
	%\arrow["{e\times \id}"', from=2-1, to=3-1]
	\arrow["{\theta_1}"', from=3-1, to=3-2]
	\arrow["\bF"', from=3-2, to=3-4]
	\arrow["{\varphi_1}"', shorten <=4pt, shorten >=4pt, Rightarrow, from=3-1, to=0]
\end{tikzcd}\]
}}

\item \textbf{Preserve Composition}  The following two pasting diagrams are equal:
\end{enumerate} 
{\footnotesize{ 
\[ \begin{tikzcd}[cramped]
	{\sP(k)\times \prod (\sP(j_r)\times \sA^{j_r})} && {\sP(k)\times \prod (\sP(j_r)\times \sB^{j_r})} \\
	& {\sP(k)\times \sA^k} && {\sP(k)\times \sB^k} \\
	{\sP(J)\times \sA^J} \\
	& \sA && {\sB}
	\arrow["{\id\times (\prod \id\times {\bF}^{j_r})}", from=1-1, to=1-3]
	\arrow["{\id\times (\prod \theta_{j_r})}"', from=1-1, to=2-2]
	\arrow["\ga"', from=1-1, to=3-1]
	\arrow["{\id\times \prod \zeta_{j_r}}", shorten <=6pt, shorten >=6pt, Rightarrow, from=1-3, to=2-2]
	\arrow["{\id\times (\prod \vartheta_{j_r})}", from=1-3, to=2-4]
	\arrow["{\id\times {\bF}^k}"', from=2-2, to=2-4]
	\arrow["\phi", shorten <=6pt, shorten >=6pt, Rightarrow, from=2-2, to=3-1]
	\arrow["{\theta_k}"', from=2-2, to=4-2]
	\arrow["{\zeta_k}", shorten <=6pt, shorten >=6pt, Rightarrow, from=2-4, to=4-2]
	\arrow["{\vartheta_k}", from=2-4, to=4-4]
	\arrow["{\theta_J}"', from=3-1, to=4-2]
	\arrow["\bF"', from=4-2, to=4-4]
\end{tikzcd}\]

equals

\[\begin{tikzcd}
	{\sP(k)\times \prod (\sP(j_r)\times \sA^{j_r})} && {\sP(k)\times \prod (\sP(j_r)\times \sB^{j_r})} \\
	&&& {\sP(k)\times \sB^k} \\
	{\sP(J)\times \sA^J} && {\sP(J)\times \sB^J} \\
	& \sA && {\sB}
	\arrow["{\id\times (\prod \id\times {\bF}^{j_r})}", from=1-1, to=1-3]
	\arrow["\ga"', from=1-1, to=3-1]
	\arrow["{\id\times (\prod \vartheta_{j_r})}", from=1-3, to=2-4]
	\arrow["\ga", from=1-3, to=3-3]
	\arrow["\varphi", shorten <=6pt, shorten >=6pt, Rightarrow, from=2-4, to=3-3]
	\arrow["{\vartheta_k}", from=2-4, to=4-4]
	\arrow["{\id\times {\bF}^J}", from=3-1, to=3-3]
	\arrow["{\theta_J}"', from=3-1, to=4-2]
	\arrow["{\zeta_J}", shorten <=6pt, shorten >=6pt, Rightarrow, from=3-3, to=4-2]
	\arrow["{\vartheta_J}", from=3-3, to=4-4]
	\arrow["\bF"', from=4-2, to=4-4]
\end{tikzcd}\]
}}
\end{defn}

\begin{defn}[$\sP$-$2$-cell]\label{2cell1} Let $(\bF,\ze), (\bG,\xi) \colon \sA \rtarr \sB$ be $\sP$-pseudomorphisms.
A $\sP$-$2$-cell $\la\colon \bF \Rightarrow \bG$, alias a $\sP$-transformation, is a transformation 
$\la\colon \bF\rtarr \bG$, not necessarily invertible, such that the following two pasting diagrams are equal.
{\footnotesize{ 
\[\begin{tikzcd}
\sP(n)\times \sA^n & & \sP(n)\times \sB^n & = & \sP(n)\times \sA^n & & \sP(n)\times \sB^n \\
 & & & & \\
 \sA & & \sB & & \sA & & \sB \\
  \arrow["{\id\times\bF^n}", curve={height=-14pt}, from=1-1, to=1-3]
  \arrow["{\id\times\bG^n}", curve={height=14pt}, from=1-1, to=1-3]
  \arrow["{\id\times\bF^n}", from=1-5, to=1-7]
  \arrow["{\tha_n}"', from=1-1, to=3-1]
  \arrow["{\tha_n}", from=1-3, to=3-3]
  \arrow["{\tha_n}"', from=1-5, to=3-5]
  \arrow["{\tha_n}", from=1-7, to=3-7]
  \arrow["{\bG}"', from=3-1, to=3-3]
  \arrow["{\bF}"', curve={height=-14pt}, from=3-5, to=3-7]
  \arrow["{\bG}"', curve={height=14pt}, from=3-5, to=3-7]
  \arrow["{\xi}"', shorten <=9pt, shorten >=9pt, Rightarrow, from=1-1, to=3-3]
  \arrow["{\ze}", shorten <=9pt, shorten >=9pt, Rightarrow, from=1-5, to=3-7]
  \arrow["{\id\times\la^n}", shorten <=3pt, shorten >=3pt, Rightarrow, from=1-1, to=1-3]
   \arrow["{\id\times\la}"', shorten <=9pt, shorten >=9pt, Rightarrow, from=3-5, to=3-7]
 \end{tikzcd}\]
}}
\end{defn}

The $\sP$-pseudoalgebras together with the
$\sP$-pseudomorphisms and $\sP$-$2$-cells form the $2$-category $\PPs$. 

\subsection{Details about $\FPs$ and $\FPR$}\label{FDetails}

We defined $\sF$-pseudoalgebras and $\sF$-pseudomorphisms in \autoref{PsFun}.  The following definition is parallel to \autoref{2cell1}, as is especially clear when $\sA$ and $\sB$ are strictly special.

\begin{defn}[$\sF$-$2$-cell]\label{2cell2} Let $(\bF,\ze), (\bG,\xi) \colon \sA \rtarr \sB$ be $\sF$-pseudomorphisms of 
$\sF$-pseudoalgebras. An $\sF$-$2$-cell, alias $\sF$-transformation, $\la\colon (\bF,\ze) \rtarr (\bG,\xi)$ is given by a transformation $\la\colon \bF \Rightarrow \bG$, not necessarily invertible, such that the following pasting diagrams are equal.   
{\footnotesize{ 
\[\begin{tikzcd}[cramped]
\sF(\bm,\bn)\times \sA_{\bm} & & \sF(\bm,\bn)\times \sB_{\bm}& = & 
\sF(\bm,\bn)\times \sA_{\bm} & & \sF(\bm,\bn)\times \sB_{\bm}\\
 & & & & \\
 \sA_{\bn} & & \sB_{\bn} & & \sA_{\bn}& & \sB_{\bn} \\
  \arrow["{\id\times\bF_{\bm}}", curve={height=-14pt}, from=1-1, to=1-3]
  \arrow["{\id\times\bG_{\bm}}", curve={height=14pt}, from=1-1, to=1-3]
  \arrow["{\id\times\bF_{\bm}}", from=1-5, to=1-7]
  \arrow["{\PS}"', from=1-1, to=3-1]
  \arrow["{\PS}", from=1-3, to=3-3]
  \arrow["{\PS}"', from=1-5, to=3-5]
  \arrow["{\PS}", from=1-7, to=3-7]
  \arrow["{\bG_{\bn}}"', from=3-1, to=3-3]
  \arrow["{\bF_{\bn}}"', curve={height=-14pt}, from=3-5, to=3-7]
  \arrow["{\bG_{\bn}}"', curve={height=14pt}, from=3-5, to=3-7]
  \arrow["{\xi}"', shorten <=9pt, shorten >=9pt, Rightarrow, from=1-1, to=3-3]
  \arrow["{\ze}", shorten <=9pt, shorten >=9pt, Rightarrow, from=1-5, to=3-7]
  \arrow["{\la}", shorten <=3pt, shorten >=3pt, Rightarrow, from=1-1, to=1-3]
   \arrow["{\la}"', shorten <=9pt, shorten >=9pt, Rightarrow, from=3-5, to=3-7]
 \end{tikzcd}\]
}}
\end{defn}

The $\sF$-pseudoalgebras together with the
$\sF$-pseudomorphisms and $\sF$-$2$-cells form the $2$-category $\FP$.  Restricting to very special or strictly special objects gives $\FPs$ and $\FPR$.  We must prove the comparison of these, as stated in
Propositions \ref{ImageOfRPSAlg} and \ref{FPsFPR}.

\begin{proof}[Proof of \autoref{ImageOfRPSAlg}]
    Let $\sA \in \FPs$ and write $\sA_\Pi$ for the restriction of $\sA$ to $\Pi$. By what it means to be very special,  $\de \colon \sA_\Pi \rtarr \bR\bL\sA_\Pi$ is an isomorphism.  We use the given $\sF$-pseudoalgebra structure on $\sA$ to extend the $\PI$-category $\bR\bL\sA_{\PI}$ to an $\sF$-pseudoalebra, which we denote by $\bR\bL\sA$.   For $\phi \in \sF(\textbf{m},\textbf{n})$, we define 
   $(\bR\bL \sA)(\ph)$ to be the evident composite
   \[ \begin{tikzcd}
    {(\bR\bL\sA_\Pi)(\textbf m)} & {\sA(\textbf m)} & {\sA(\textbf n)} & {(\bR\bL\sA_\Pi)(\textbf n).}
    \arrow["{\de^{-1}}", from=1-1, to=1-2]
    \arrow["{\sA(\phi)}", from=1-2, to=1-3]
    \arrow["{\de}", from=1-3, to=1-4]
    \end{tikzcd}\]
If $\phi \in \Pi(\textbf{m},\textbf{n})$, then this is $(\bR\bL\sA_\Pi)(\phi)$ since $\de$ is an isomorphism of $\PI$-categories.     To check that this gives a pseudoalgebra, we need to check that it gives  a pseudofunctor $\sF \rtarr \textbf{Cat}_*$, but this follows directly from the pseudofunctoriality of $\sA$. Intuitively, we are whiskering everything from left and right by $\de^{-1}$ and $\de$.  The pseudoalgebra  $\bR\bL\sA$ is isomorphic to $\sA$ since the following  commutative diagram for 
$f \in \sF(\textbf{m},\textbf{n})$ shows that the isomorphism $\de$ is a pseudomorphism.
    \[\begin{tikzcd}
    {\sA(\textbf m)} &&& {\sA(\textbf n)} \\
    {\bR\bL(\sA_\Pi)(\textbf m)} & {\sA(\textbf m)} & {\sA(\textbf n)} & {\bR\bL(\sA_\Pi)(\textbf n)}
    \arrow["{\sA(f)}", from=1-1, to=1-4]
    \arrow["{\de}"', from=1-1, to=2-1]
    \arrow["{\xi_n}", from=1-4, to=2-4]
    \arrow["{\de^{-1}}", from=2-1, to=2-2]
    \arrow["{\sA(f)}", from=2-2, to=2-3]
    \arrow["{\de}", from=2-3, to=2-4]
    \end{tikzcd}\]
\end{proof}

\begin{proof}[Proof of \autoref{FPsFPR}]  We claim that the construction $\bR\bL$  of the previous proof gives the map on $0$-cells  of a $2$-functor $\bR\bL\colon \FPs\rtarr \FPR$ and that  
$\xi\colon \sA \rtarr \bR\bL \sA$ gives an isomorphism from the identity $2$-functor of $\FPs$ to 
$\bI\bR\bL$, where $\bI \colon \FPR \rtarr \FPs$ is the inclusion.  For functoriality on $1$-cells,  let 
$\bF\colon \sA \rtarr \sB$ be an $\sF$-pseudomorphism of $\sF$-pseudoalgebras.  We define
$\bR\bL\bF\colon \bR\bL \sA \rtarr \bR\bL \sB$ on objects as the composite 
    \[\begin{tikzcd}
    {(\bR\bL\sA)(\textbf n)} & {\sA(\textbf n)} & {\sB(\textbf n)} & {(\bR\bL\sB)(\textbf n)}
    \arrow["{\de^{-1}}", from=1-1, to=1-2]
    \arrow["{\bF(\bn)}", from=1-2, to=1-3]
    \arrow["{\de}", from=1-3, to=1-4]
    \end{tikzcd}\]
    We need to check that $\bR\bL \bF$ defines a pseudotransformation.  For $\ph\in \sF(\textbf{m},\textbf{n})$,  the left and right squares of the following commute:
    \[\begin{tikzcd}
    {(\bR\bL\sA)(\textbf m)} & {\sA(\textbf m)} & {\sB(\textbf m)} & {(\bR\bL\sB)(\textbf m)} \\
    {(\bR\bL\sA)(\textbf n)} & {\sA(\textbf n)} & {\sB(\textbf n)} & {(\bR\bL\sB)(\textbf n)}
    \arrow["{\xi_m^{-1}}", from=1-1, to=1-2]
    \arrow["{\bR\bL\sA(f)}"', from=1-1, to=2-1]
    \arrow["{\bF(\bm)}", from=1-2, to=1-3]
    \arrow["{\sA(f)}"', from=1-2, to=2-2]
    \arrow["{\de}", from=1-3, to=1-4]
    \arrow["{\phi_f}"', shorten <=6pt, shorten >=6pt, Rightarrow, from=1-3, to=2-2]
    \arrow["{\sB(f)}", from=1-3, to=2-3]
    \arrow["{\bR\bL\sB(f)}", from=1-4, to=2-4]
    \arrow["{(\de^{-1}}", from=2-1, to=2-2]
    \arrow["{\bF(\bn)}", from=2-2, to=2-3]
    \arrow["{\de}", from=2-3, to=2-4]
    \end{tikzcd}\]
Thus  $\phi_f$ for $\bR\bL\bF$ is  $\de^{-1}\phi_f \de$. As before, since whiskering by $\de$ commutes with everything in sight, all of the coherence conditions hold automatically. Compatibility with $2$-cells is similar.  Since $\bR\bL \colon \FPs \rtarr \FPR$ restricts to the identity on $\FPR$, it is left inverse to the inclusion $\bI\colon \FPR \rtarr \FPs$.  Therefore $\bR\bL$ and $\bI$ are inverse
equivalences of $2$-categories.
\end{proof}

\subsection{Details of the $2$-functor $\bQ$}\label{Qdetails}

We prove Propositions \ref{QProp1} and \ref{QProp2} here.   For simplicity of notation, we write $\sE(n)$ for $\sE(\textbf{n},\textbf{1})$. 

\begin{proof}[Proof of \autoref{QProp1}]
For $\sB\in \FPR$ with $\sB(1) =\sA$, we have $\bQ\sB = \sA$ with structure maps $\tha = \tha_n\colon \sA^n \rtarr \sA$   as defined in \autoref{SecQ}.  We must prove that $\sA$ is a $\sP$-pseudoalgebra.
Thus we must show that our $\tha_n$ satisfy the properties specified in \autoref{PPsDefn}.   

We write $\PS$ for the action maps $\sF(\bm,\bn)\times \sA^m \rtarr \sA^n$, calling them $\PS_m$ when $n=1$.  In particular, identifying $\ph_1\in \sE(1)$ with $e_1\in \sP(1)$, we defined $\tha_1$ to be 
$\PS_1$.  Therefore the identity isomorphism of the pseudofunctor $\sB$ gives the identity isomorphism $\ph_1$ in \autoref{phione}.  Similarly, when restricting to canonical products by focusing on the objects $e_k$ and $e_{j_i}$ of $\sP(k)$ and the $\sP(j_i)$, the associativity isomorphism $\ph$ in \autoref{phikey} is given by the composition isomorphism of the pseudofunctor $\sB$.  Extending the $\tha_n$ and $\ph$ by the equivariance formulas in \autoref{PPsDefn}, which are dictated by the compositional equivariance of the pseudofunctor $\sB$, this outlines the construction of $\tha$ and the $\ph$. 

We give a better idea of the associativity diagram \autoref{phikey} by comparing it with the restriction to 
$\sE$ of an analogous diagram for the action of  $\sF$ on $\sB$, defined using the product operation that we used in \autoref{wedprod}. 

\begin{equation}\label{phikeytoo}
\begin{tikzcd}
	{\sE(k)\times (\prod_i\sE(j_i)\times \sA^{j_i})} & {\sE(k)\times \sA^k} \\
	& \sA \\
	{\sE(k)\times \prod_i \sE(j_i)\times \sA^{j_+}} & {\sE(j_+)\times \sA^{j_+}}
	\arrow["\id\times \prod_i\PS_{j_i}", from=1-1, to=1-2]
	\arrow["{\text{shuffle}}"', from=1-1, to=3-1]
	\arrow["\PS_k", from=1-2, to=2-2]
	\arrow["{\phi}", shorten <=6pt, shorten >=6pt, Rightarrow, from=1-2, to=3-1]
	\arrow["\ga\times \id"', from=3-1, to=3-2]
	\arrow["\PS_{j_+}"', from=3-2, to=2-2]
\end{tikzcd}
\end{equation}
Here $\ga$ is the precursor in $\sE$ of the structure maps $\ga$ of the operad $\sP$. As is implicit in the comparisons in \autoref{SecQ}, the latter are constructed from the former by use of $\LA(\bn,\bn) = \SI_n$ and composition in $\sF$.   Looking just at the canonical products $\ph_n \in \sE(\bn)$, this diagram commutes up to a natural isomorphism $\ph$ by \autoref{wedprod} and the associativity isomorphism of the pseudofunctor $\sB$.  In particular, this gives the diagram \autoref{phikey} on the objects $e_n\in \sP(n)$.   We can use \autoref{cando} to relate the diagrams \autoref{phikey} and 
\autoref{phikeytoo} via the following diagram.
{\small{
    \begin{equation}\label{libel}
    \begin{tikzcd}
    {\sP(k)\times\prod \sP(j_i)\times\sA^{j_i}} &&& {\sP(k)\times\sA^k} \\
    & {\sE(k)\times\prod\sE(j_i)\times\sA^{j_i}} & {\sE(k)\times\sA^k} \\
    \\
    & {\sE(j_+)\times\sA^{j_+}} & \sA \\
    {\sP(j_+)\times\sA^{j_+}} &&& \sA
    \arrow[""{name=0, anchor=center, inner sep=0}, from=1-1, to=1-4]
    \arrow["proj"', from=1-1, to=2-2]
    \arrow[from=1-1, to=5-1]
    \arrow["proj", from=1-4, to=2-3]
    \arrow[""{name=1, anchor=center, inner sep=0}, from=1-4, to=5-4]
    \arrow[""{name=2, anchor=center, inner sep=0}, from=2-2, to=2-3]
    \arrow[from=2-2, to=4-2]
    \arrow[""{name=3, anchor=center, inner sep=0}, from=2-3, to=4-3]
    \arrow[""{name=4, anchor=center, inner sep=0}, from=4-2, to=4-3]
    \arrow["proj"', from=5-1, to=4-2]
    \arrow[""{name=5, anchor=center, inner sep=0}, from=5-1, to=5-4]
    \arrow[equals, from=5-4, to=4-3]
    \arrow[shorten <=6pt, shorten >=6pt, Rightarrow, from=0, to=2]
    \arrow[shorten <=6pt, shorten >=6pt, Rightarrow, from=1, to=3]
    \arrow["{\phi}"', shorten <=6pt, shorten >=6pt, Rightarrow, from=2, to=4]
    \arrow[shorten <=6pt, shorten >=6pt, Rightarrow, from=5, to=4]
    \end{tikzcd}
 \end{equation}
  }}
The outer and inner and squares are rewritings of \autoref{phikey} and \autoref{phikeytoo}.  The left trapezoid commutes trivially. The bottom and right trapezoids are instances of \autoref{cando}.  The top trapezoid  is given by $k$ instances of \autoref{cando}, glued together via \autoref{wedprod}.
%In particular, adjacent faces begin with the same natural isomorphism. E.g., along the arrow $\theta\colon \sP(m_+)\times\sA^m \rtarr \sA$, both adjacent faces begin with the natural isomorphism $\phi_{m_+} \colon \theta \implies \circ(proj)$, and along the edge $\gamma \times 1 \colon \sP(n)\times\prod\sP(m_i)\times\sA^m$, both adjacent faces begin with the identity natural isomorphism. What this lets us do is paste these inset squares together, to form a smaller central cube in the pasting diagram. Since both cubes finish at $\sB$, each edge composite in the pasting diagram factors through our smaller cube. 

We need to check the  Operadic Composition law in \autoref{PPsDefn}, which states that the two pasting diagrams displayed there are equal.  We can replace all $\sP(n)$ by $\sE(n)$ and all $\tha_n$ by $\PS_n$ in those diagrams.  Remember that $\ga$ in those diagrams  is always a composite of structure maps  $\ga\times \id$ and shuffle isomorphisms.   There result considerably simpler diagrams that just involve canonical products in $\sF$-pseudoalgebras and which are easily checked to be equal by hand. Checking that our diagrams compose properly, we see that Operadic Composition for the pasting diagrams of the 
$\sF$-pseudoalgebra $\sB$ imply Operadic composition for the pasting diagrams of $\bQ\sB$.  Similar but simpler arguments show that the implicit Operadic Identity and Composition law for the $\sF$-pseudoalgebra $\sB $ implies its analog for $\bQ\sB$, as specified in \autoref{PPsDefn}. 
\end{proof}

\begin{proof}[Proof of \autoref{QProp2}]   We have defined $\bQ$ on objects ($0$-cells) and must extend the definition to $1$-cells and $2$-cells.  Thus  let $\et\colon \sB \rtarr \sB'$ be a $1$-cell between objects $\sB$ and $\sB'$ of $\FPR$, where $\sB = \bR\sA$ and $\sB' = \bR\sA'$ as $\PI$-categories.  The pseudomorphism $\et$ must be given by a map $\et\colon \bR\sA \rtarr \bR\sA'$ of $\PI$-categories. We claim that $\et\colon \sA \rtarr \sA'$ gives a map, $\bQ\et$, of $\sP$-pseudoalgebras.  This just means that our construction of the action functors $\tha\colon \sP(n) \times \sA^n \rtarr \sA$ from the functor $\sB\colon \sF\rtarr \mathbf{Cat}$ is pseudonatural.  Thus we must have naturality diagrams
    \[\begin{tikzcd}
    	{\sP(n)\times \sA^n} & \sA \\
    	{\sP(n)\times (\sA')^n} & \sA'
    	\arrow[""{name=0, anchor=center, inner sep=0}, "\theta", from=1-1, to=1-2]
    	\arrow["{1\times \eta^n}"', from=1-1, to=2-1]
    	\arrow["{\eta}", from=1-2, to=2-2]
    	\arrow[""{name=1, anchor=center, inner sep=0}, "\theta"', from=2-1, to=2-2]
    	\arrow[shorten <=6pt, shorten >=6pt, Rightarrow, from=0, to=1]
    \end{tikzcd}\]
    Modulo further details needed for the construction on  morphisms on the left, this is a consequence of the following diagram:
    \[\begin{tikzcd}
    	{\sE(\textbf n,\textbf 1) \times\LA(\textbf n, \textbf n) \times \sA^n} & {\sE(\textbf n,\textbf 1) \times \sA^n} & \sA \\
    	{\sE(\textbf n,\textbf 1) \times\LA(\textbf n, \textbf n) \times (\sA')^n} & {\sE(\textbf n,\textbf 1) \times (\sA')^n} & \sA'
    	\arrow["{1\times\PS}", from=1-1, to=1-2]
    	\arrow["{1\times1\times\eta^n}", from=1-1, to=2-1]
    	\arrow[""{name=0, anchor=center, inner sep=0}, "\PS", from=1-2, to=1-3]
    	\arrow["{1\times\eta^n}", from=1-2, to=2-2]
    	\arrow["{\eta}", from=1-3, to=2-3]
    	\arrow["{1\times\PS}"', from=2-1, to=2-2]
    	\arrow[""{name=1, anchor=center, inner sep=0}, "\PS"', from=2-2, to=2-3]
    	\arrow[shorten <=9pt, shorten >=9pt, Rightarrow, from=0, to=1]
    \end{tikzcd}\]
Given $\et'\colon \sA' \rtarr \sA"$, we can compose diagrams to check that composition behaves correctly,
$(\bQ\et')(\bQ\et) = \bQ(\et' \et)$, coherently.  The analogous verifications for $2$-cells are similar inspections of definitions.
\end{proof}

\subsection{Details of the $2$-functor $\bR$}\label{Rdetails}

\begin{proof}[Proof of \autoref{RProp1}] For a $\sP$-pseudoalgebra $\sA$, we must prove that $\bR\sA$ is an $\sF$-pseudoalgebra.  Since $\bR\sA$ is strictly special, we only need to show that $\bR\sA$ is a pseudofunctor.  The identity isomorphism required by \autoref{PsFun} is immediate from the identity transformation \autoref{phione} and is unnecessary in the normal variant of Remarks \ref{norm1} and \ref{norm2}.  We must check the composition isomorphism required by \autoref{PsFun}.   That is, we
must check that $\bR\sA(\psi\phi)$ is naturally isomorphic to $(\bR\sA\psi)(\bR\sA\phi)$, where $\phi\colon \textbf{m}\rtarr\textbf{n}$ and $\psi\colon \textbf{n}\rtarr\textbf{k}$ are maps in $\sF$.  Dealing with projections separately as in the construction of $\bR$, we can assume that $\ph$ and $\ps$ are in $\sE$.  

As in \autoref{decompose}, decompose $\phi$ and $\psi$ as composites of a permutation and a wedge of canonical product maps:
    \[\begin{tikzcd}
    	{\textbf m} & {\textbf m} && {\textbf n} & {\textbf n} \\
    	& {\textbf n} &&& {\textbf k}
    	\arrow["\sigma", from=1-1, to=1-2]
    	\arrow["\phi"', from=1-1, to=2-2]
    	\arrow["{\vee \phi_{a_i}}", from=1-2, to=2-2]
    	\arrow["\tau", from=1-4, to=1-5]
    	\arrow["\psi"', from=1-4, to=2-5]
    	\arrow["{\vee \phi_{b_i}}", from=1-5, to=2-5]
    \end{tikzcd}\]
For simplicity we may assume that $k=1$; the general case follows by use of products as in 
\autoref{wedprod}.   We can decompose $\ps\com\ph$ similarly.
    \[\begin{tikzcd}
    	{\textbf m} & {\textbf m} & {\textbf n} & {\textbf n} \\
    	& {\textbf m} && {\textbf 1}
    	\arrow["\sigma", from=1-1, to=1-2]
    	\arrow["{\vee \phi_{a_i}}", from=1-2, to=1-3]
    	\arrow["\rho", from=1-2, to=2-2]
    	\arrow["\tau", from=1-3, to=1-4]
    	\arrow["{\phi_n}", from=1-4, to=2-4]
    	\arrow["{\phi_m}"', from=2-2, to=2-4]
    \end{tikzcd}\]
Note that we can take $\rho$ to be the block permutation of $m = a_1+ \cdots + a_n$ letters determined by $\ta$. (Some $a_i$ may be zero.). On applying $\bR$, the top composite is $(\bR\sA\ps)(\bR\sA \ph)$ and the bottom composite is $\bR\sA(\ps\ph)$ in the resulting diagram
\[\begin{tikzcd}
    	{\sA^m} & {\sA^m} & {\sA^n} & {\sA^n} \\
    	&{\sA^m} && {\sA}
    	\arrow["\sigma", from=1-1, to=1-2]
    	\arrow["{\prod \phi_{a_i}}", from=1-2, to=1-3]
	\arrow["\rho", from=1-2, to=2-2]
    	\arrow["\tau", from=1-3, to=1-4]
    	\arrow["{\phi_n}", from=1-4, to=2-4]
	\arrow["{\phi_m}"', from=2-2, to=2-4]
	\arrow["{\phi}", shorten <=6pt, shorten >=6pt, Rightarrow, from=1-4, to=2-2]. 
    \end{tikzcd}\]
Up to permutations, on the top we are taking the canonical $n$-fold product of the canonical $a_i$-fold products, $1\leq i\leq n$, of the $m$-variables; on the bottom, we are directly taking their $m$-fold product .  These agree up to specified permutation, and the required natural isomorphism z$\ph$ is given by \autoref{FPRdiag2}.

To complete the check that $\bR\sA$ is a pseudofunctor, we must check coherence conditions for the composition just defined.  Unitality has two variants, as in Remarks \ref{norm1} and \ref{norm2}, but is direct from the unitality of $\sA$.  Associativity up to coherent isomorphism requires a check.  Changing notation away from Greek letters for clarity, we must check associativity for a composable triple $(f, g, h)$ of maps in $\sF$.  Writing $\phi_{f,g}$ for the natural isomorphism $\bR\sA(f)\bR\sA(g)\implies \bR\sA(fg)$ just constructed and writing $\ast$ for whiskering, we need to show that the following diagram commutes:
\[\begin{tikzcd}
    {(\bR\sA h\circ \bR\sA g) \circ \bR\sA f} & {\bR\sA h\circ (\bR\sA g \circ \bR\sA f)} \\
    {\bR\sA (h g) \circ \bR\sA f} & {\bR\sA h \circ \bR\sA (gf)} \\
    {\bR\sA((hg)f)} & {\bR\sA(h(gf))}
    \arrow[Rightarrow, from=1-1, to=1-2]
    \arrow["{\phi_{h,g}*\id}"', Rightarrow, from=1-1, to=2-1]
    \arrow["{\id * \phi_{g,f}}", Rightarrow, from=1-2, to=2-2]
    \arrow["{\phi_{hg,f}}"', Rightarrow, from=2-1, to=3-1]
    \arrow["{\phi_{h,gf}}", Rightarrow, from=2-2, to=3-2]
    \arrow[Rightarrow, from=3-1, to=3-2]
    \end{tikzcd}\]
    Details of the rather horrendous diagrams needed are in \cite[Section 3]{LRZZ2} and will not be repeated here.\footnote{Except for details in this subsection and the next, everything in \cite{LRZZ, LRZZ2} has been reworked in this paper, including corrections of a few mistakes in those sources.  The diagrammatic proofs omitted here are worked out correctly in those two sources in work that is entirely due to the REU participant coauthors.}. In brief, writing out the top and bottom composites explicitly and chasing diagrams that are otherwise strictly commutative, we find that the required equality is given by special cases of the equality of Operadic Composition pasting diagrams that is a key part of the definition of a $\sP$-pseudoalgebra \autoref{PPsDefn}.
\end{proof}

\begin{proof}[Proof of \autoref{RProp2}]  We have defined $\bR$ on objects ($0$-cells) and must extend the definition to $1$-cells and $2$-cells.  Thus  let $(\bF,\ze)\colon \sA \rtarr \sB$ be a $\sP$-pseudomorphism, as defined in \autoref{pseudomap}.  We must extend the evident morphism $
\bR\colon \bR\sA \rtarr \bR\sB$ of $\PI$-categories to a pseudomorphism of pseudofunctors 
$\sF\rtarr \mathbf{Cat}$.   Thus we must   show that each diagram
 \[\begin{tikzcd}
\sF(\bm,\bn)\times \sA^m  &  \sF(\bm,\bn)\times \sB^m. \\
\sA^n & \sB^n 
\arrow["\id\times \bF^m", from=1-1, to=1-2]
\arrow["\PS", from=1-1, to=2-1]
\arrow["\PS", from=1-2, to=2-2]
\arrow["\bF^n"', from=2-1, to=2-2]
\arrow["\ze",  Rightarrow, from=1-2, to = 2-1]
   \end{tikzcd}\]
commutes up to a natural isomorphism $\ze$.  Since all maps in $\sF$ are composites of maps in $\PI$ and wedges of canonical products, it suffices (using \autoref{wedprod}) to define $\ze$ when $m=n$, $k=1$, and we restrict to $\ph_n\in\sE(\bn,\mathbf{1})$.   But then we are looking at the canonical product, which is given by $\tha_n$ restricted to $e_n$, hence $\ze$ is given directly from the $\ze$ of $\bF$. 

    We must still check that the appropriate coherence conditions are satisfied for a pseudotransformation. We  label the natural isomorphisms for $\bR\sA$ as $\alpha$, and those for $\bR\sB$ as $\beta$.  For $f \in \sF(\textbf{m},\textbf{n})$, we label our natural isomorphism $F^n\bR\sA(f) \implies \bR\sB(f)F^m$ as $\zeta_f$.  Then we must show the following equality of pasting diagrams:
    \[\begin{tikzcd}
    {\sA^m} && {\sA^k} && {\sA^m} & {\sA^n} & {\sA^k} \\
    &&& {=} \\
    {\sB^m} & {\sB^n} & {\sB^k} && {\sB^m} & {\sB^n} & {\sB^k}
    \arrow[""{name=0, anchor=center, inner sep=0}, "{\bR\sA(gf)}", from=1-1, to=1-3]
    \arrow["{F^m}"', from=1-1, to=3-1]
    \arrow["{F^k}", from=1-3, to=3-3]
    \arrow["{\bR\sA(f)}", from=1-5, to=1-6]
    \arrow[""{name=1, anchor=center, inner sep=0}, "{\bR\sA(gf)}", curve={height=-24pt}, from=1-5, to=1-7]
    \arrow[""{name=2, anchor=center, inner sep=0}, "{F^m}"', from=1-5, to=3-5]
    \arrow["{\bR\sA(g)}", from=1-6, to=1-7]
    \arrow[""{name=3, anchor=center, inner sep=0}, "{F^n}", from=1-6, to=3-6]
    \arrow[""{name=4, anchor=center, inner sep=0}, "{F^k}", from=1-7, to=3-7]
    \arrow["{\bR\sB(f)}"', from=3-1, to=3-2]
    \arrow[""{name=5, anchor=center, inner sep=0}, "{\bR\sB(gf)}", curve={height=-24pt}, from=3-1, to=3-3]
    \arrow["{\bR\sB(g)}"', from=3-2, to=3-3]
    \arrow["{\bR\sB(f)}"', from=3-5, to=3-6]
    \arrow["{\bR\sB(g)}"', from=3-6, to=3-7]
    \arrow["{\zeta_{gf}}"', shorten <=8pt, shorten >=10pt, Rightarrow, from=0, to=5]
    \arrow["{\zeta_f}"', shorten <=10pt, shorten >=10pt, Rightarrow, from=2, to=3]
    \arrow["\phi"', Rightarrow, from=1-6, to=1]
    \arrow["{\zeta_g}"', shorten <=14pt, shorten >=8pt, Rightarrow, from=3, to=4]
    \arrow["\psi", Rightarrow, from=3-2, to=5]
    \end{tikzcd}\]
Another rather horrendous pair of pasting diagrams, also in \cite[Section 3]{LRZZ} and not repeated here, reduced this equality to an equivalent equality of pasting diagrams  for $\bF$ as a pseudomorphism of 
 $\sP$-pseudoalgebras. 
  \end{proof}
  
\subsection{Details of the $2$-functor $\bU$}\label{Udetails}

\begin{proof} [Proof of \autoref{UProp1}] For a $\sP$-pseudoalgebra $\sA$, we defined $\bU \sA$ by giving $\sA$ the canonical product $\otimes$ induced by $\tha_2$ acting on the product of $e_2\in \sP(2)$  with $\sA^2$ and the unit object induced similarly by $e_o\in \sP(0)$.  The identity operation is induced by $e_1\in\sP(1)$.  In line with \autoref{norm1}, it is convenient to insist that $\sA$ be normal.  We must show that $\bU \sA = (\sA,\otimes,e)$ is symmetric monoidal.  The idea is that the key associativity diagram, \autoref{phikey}, is a simultaneous generalization of the defining properties of a symmetric monoidal category.  Since full details are spelled out  in \cite[Section 3]{LRZZ},\footnote{Our $\bU$ and $\bV$ are denoted $\bF$ and $\bG$ there.} we just say which case of that diagram we are applying to construct the coherence maps and which coherence properties of that diagram we are using to obtain the coherence diagrams of $\sA$ as a symmetric monoidal category. We use the notation $(k;j_1, \cdots, j_k)$ to describe that case of \autoref{phikey}. 

\begin{enumerate}[(i)]
    \item \textbf{(Unit)}. The left and right unitor isomorphisms $\la$ and $\rh$ come from the cases $(2;0,1)$ and $(2;1,0)$ of \autoref{phikey}, using the objects $e_2$, $e_0$, and $e_1$ in $\sP(2)$, 
$\sP(0)$, and $\sP(1)$. 
    \item \textbf{(Associativity)}. The associativity isomorphism $\al$ is the composite 
    $$(a\otimes b)\otimes c \iso a\otimes b\otimes c \iso a\otimes(b\otimes c)$$
of isomorphisms with the canonical $3$-fold product given by the cases $(2;2,1)$ and $(2;1,2)$ of \autoref{phikey}, using the objects $e_2$ and $e_1$ of $\sP(2)$ and $\sP(1)$. These two diagrams (the second one flipped) can be glued together to give a single associativity diagram.
\item \textbf{(Braiding)}.  The braiding $\be\colon a\otimes b \iso b\otimes a$ comes from the equivariance of $\tha_2$ and the isomorphism $\si \rtarr e_2$ in $\sP(2)$, where $\si = (12) \in \SI_2$:
$$a\otimes b = \tha_2(e_2; a, b) = \tha_2(\si; b, a)  \rtarr \tha_2(e_2; b,a) = b\otimes a$$
\end{enumerate}
We check the coherence diagrams in order.
\begin{enumerate}[(i)]
\item \textbf{(Pentagon)}. There is a canonical isomorphism between each of the neighbors in the pentagon diagram with the canonical $4$-fold product and therefore with each other.  These are obtained by composing isomorphisms given by \autoref{phikey}. For example
$$  (a\otimes (b\otimes c))\otimes d \iso (a\otimes b\otimes c)\otimes d \iso a\otimes b\otimes c\otimes d $$
and
$$ ((a\otimes b)\otimes c)\otimes d\iso (a\otimes b\otimes c)\otimes d  \iso a\otimes b\otimes c\otimes d$$
With these composite isomorphisms between neighbors, the analog of the pentagon diagram commutes by direct inspection, using that at each vertex these composites  are both canonically  isomorphic to  the canonical $4$-fold product.  Applications of the Operadic Composition equality of pasting diagrams then shows that  these composite isomorphisms between neighboring vertices are the same as the associativity isomorphisms that appear in the  pentagon diagram.
\item \textbf{(Triangle)}. Since $\mathbb{U}\sA$ satisfies the strict unit law, the triangle diagram is clearly commutative.
\item \textbf{(Hexagon)}. Using equivariance and associativity, all six terms in the hexagon diagram for $\bU\sA$ are of the form  $\tha_3(\ta; a, b, c)$ for some $\ta\in \SI_3$ and with $(a, b, c)$ occurring in that order. Writing out the hexagon in that equivalent way, we see two $3$-fold  composites $e_3 \rtarr e_3$ in the $\sP(3)$ variable.  These composites must both be the identity map. Therefore, applying $\tha_3$ to that diagram we obtain the identity isomorphism, showing that the hexagon axiom holds.
\item \textbf{(Symmetry)}. Since $\be$ is induced by $e_2 \rtarr \ph$, $\be^{-1}$ is induced by $\ph \to e_2$.  By the equivariance of $\tha_2$, we see that $\be^{-1} = \be$. 
\end{enumerate}
\end{proof}

\begin{proof}[Proof of \autoref{UProp2}] We have defined $\bU$ on objects ($0$-cells) and must extend the definition to $1$-cells and $2$-cells to show that $\bU$ is a $2$-functor.  The main point is to make precise what $\bU$ does on $1$-cells.  Thus let $(\bF,\ze)\colon \sA\rightarrow \sB$ be a pseudomorphism of $\sP$-pseudoalgebras.   Then $\ze_0$ gives an isomorphism $\epz\colon e_{\sB}\rtarr \bF(e_{\sA})$ and the restriction of $\ze_2$ to $e_2\times \sA^2$ gives an isomorphism  
$$\mu\colon \bF(a_1)\otimes \bF(a_2)\rtarr  \bF(a_1\otimes a_2).$$
We claim that these give a symmetric monoidal functor $\bU\bF\colon \bU\sA \rtarr \bU\sB$.  We must check associativity, unitality, and symmetry, as formulated in \autoref{SMMap}.

\begin{enumerate}[(i)]
    \item \textbf{(Associativity)}. Consider the Operadic Composition pasting diagrams in \autoref{PPsDefn}. The cases $(2;2,1)$ and $(2;1,2)$ equate the two vertical composites of the associativity diagram with
    $$\ze_3\colon \bF\sA\otimes \bF\sA \otimes \bF\sA \rtarr \bf(\sA\otimes \sA \otimes \sA.$$
 In view of how the associativity isomorphisms for $\bU\sA$ and $\bU\sB$ are constructed, an immediate comparison of diagrams gives the commutativity of the associativity diagram.
\item \textbf{(Unitality)}.  In view of \autoref{eunit}, this is a check from the identity relation for $\bF$ as a $\sP$-pseudomorphism that is given in \autoref{pseudomap}. 
\item \textbf{(Symmetry)}.   Because $\ze_2$ is natural with respect to $\PI$, it is equivariant, and it is also natural with respect to morphisms in $\sP(2)$.  In view of how $\be$ is defined, this implies commutativity of the symmetry diagram.
\end{enumerate} 
Modulo notation, the definition of $\bU$ on $2$-cells is a specialization of \autoref{2cell2}.  We restrict $\la$ there to the case $n=2$ of its pasting diagrams to see that $\la$ commutes with products as required by
\autoref{SMMap2};  it preserves unit objects by the case $n=0$.
\end{proof}

\subsection{Details of the $2$-functor $\bV$}\label{Vdetails}

\begin{proof}[Proof of \autoref{VProp1}]  For a symmetric monoidal category $\sA$, we must define a $\sP$-pseudoalgebra $\bV\sA$.  In \autoref{V}, we specified that the underlying category of $\bV\sA$ is $\sA$ and we defined the $\SI_j$-equivariant action functors $\tha\colon \sP(n)\times \sA^n \rtarr \sA$, by first choosing an iterated canonical $n$-fold tensor product to give $\tha_n$ on $e_n\times \sA^n$ and then requiring equivariance.   

The unit isomorphism $\ph_1$ required by \autoref{phione} is given by the identity; that is, 
${\ph_1}_{e_1, a} = \id_{a}$.  We must specify the isomorphism $\ph = \ph(k;j_1, \cdots, j_k)$ required by 
\autoref{phikey}. By the equivariance formulas required of $\ph$, if suffices to define $\ph$ on objects 
$(e_k; e_{j_1},\cdots, e_{j_k})\times \bar{a}$ where $\bar{a}$ is an $n$-tuple of objects of $\sA$.   Going around the top of \autoref{phikey}, we are first taking the canonical products of each of the $k$ blocks of $j_i$ objects of $\sA$ and then taking the product of those resulting products.  Going around the bottom, we are directly taking the product of these $j_+$ objects of $\sA$.   By associativity, the first is canonically isomorphic to the second, giving us $\ph$. 

We must check coherence, showing that the equalities of pasting diagrams specified in Operadic Composition and Operadic Identity and Composition all hold.  But that is what MacLane's coherence theorem tells us!  Each    
required equality is one of well-formed diagrams, and there is a unique isomorphism from source to target constructed from the coherence constraints $(\al, \be, \la, \rh)$ of the given symmetric monoidal structure on 
$\sA$. 
\end{proof}

\begin{proof}[Proof of \autoref{VProp2}]  We must show that $\sV$ extends to a $2$-functor. Again, the key point is to define $\bV$ on a symmetric monoidal functor $(\bF,\epz, \mu)\colon \sA \rtarr \sB$ of symmetric monoidal categories.   Of course, the underlying functor of $\bV\bF$ is $\bF$.  

We must define the natural transformations $\ze$ required in \autoref{Fzeta}.  By equivariance, it suffices to define them on objects $e_j\times \bar{a}$.  We start with 
$$\ze_0= \epz\colon e_{\sB} \rtarr \bF(e_{\sA}), \ \ {\ze_1}_{e_1, a} = \id_{\bF(a)}$$ 
and 
$${\ze_2}_{(e_2,a_1, a_2)} = \mu\colon \bF(a_1)\otimes \bF(a_2) \rtarr \bF(a_1\otimes a_2).$$
Given our inductive definition of the canonical product $a_1\otimes \cdots\otimes a_n$, we then define
$${\ze_j}_{(e_j,a_1, \cdots, a_j)} = \mu\colon \bF(a_1\otimes \cdots \otimes a_{j-1}) \otimes \bF(a_j) 
\rtarr \bF(a_1\otimes \cdots \otimes a_j)$$
The equality of pasting diagrams is again a consequence of MacLane's coherence theorem, this time applied to morphisms of symmetric monoidal categories.

Similarly, a natural transformation of symmetric monoidal categories $\xi\colon \bF \rtarr \bG$, as in \autoref{SMMap2},  gives a $\sP$-$2$-cell $\la\colon \bF \Rightarrow \bG$, as in \autoref{2cell1}.
\end{proof}

\section{Interpretation in terms of $2$-monads?}\label{Sec5}

\subsection{$2$-monads and their algebras} 

Behind the scenes, this is in part a paper all about $2$-monads, yet we have little of real substance to say.  We will just briefly indicate the context,  correct a small mistake in \cite{GMMO2},  and raise the  question of how our work relates to the $2$-monadic descriptions of $\PPs$ and $\FP$.

A $2$-monad in a ground $2$-category $\sK$ is a $\bf{Cat}$-enriched monad in $\sK$.   In our work,  $\sK$ is $\bf{Cat}$.  While $2$-monads give a very convenient common language for strictification functors, they can be viewed as peripheral contextualization here.   In particular, any use of colimits in $\sK$ is purely formal contextualization. 

As in \cite{GMMO1}, for a $2$-monad $\bT$ in $\sK$, there is a notion of a $\bT$-pseudoalgebra, of a $\bT$-pseudomorphism between $\bT$-pseudoalgebras, and of an algebra $2$-cell between $\bT$-pseudoalgebras.    There we only considered normal $\bT$-pseudoalgebras, with the same motivation from symmetric monoidal categories as here.  A $\bT$-pseudoalgebra comes with invertible $2$-cells 
$$\io \colon \Id \implies \tha\com \et \ \ \text{and} \ \ \varphi\colon \tha\com \bT\tha \implies \tha \com \mu$$ 
 that satisfy coherence conditions that are spelled out, for example, in \cite[2.4]{Power}.   With these definitions, we have $2$-categories 
$$ \bT\text{-}\bf{PsAlg}, \ \  \bT\text{-}\bf{AlgPs}, \ \ \text{and} \ \ \bT\text{-}\bf{AlgSt}$$
of $\bT$-pseudoalgebras and pseudomorphisms or of strict $\bT$-algebras and either pseudomorphisms or strict morphisms (the last being only a $1$-category, so that its algebra $2$-cells are identities). 

\subsection{The $2$-monad $\bP$}
Just as in the original $1$-categorical context, there is no problem constructing a $2$-monad $\bP$ in $\mathbf{Cat}_*$ and proving the following theorem. 

\begin{thm}\label{thmP} The $2$-category $\PPs$ is isomorphic to the $2$-category $\bP\text{-}\bf{PsAlg}$.
\end{thm}

On objects, $\bP\sA = \sqcup_{n\geq 0} \sP(n) \times_{\SI_n} \sA^n$, the component at $n=0$ giving the base object.  Formally, this is the categorical tensor product $\sP\otimes_{\SI} \sA^*$ of the contravariant functor 
$\sP\colon \SI \rtarr \mathbf{Cat}$ and the covariant functor $\sA^*\colon \SI \rtarr \mathbf{Cat}_*$ given by the cartesian powers $\sA^n$.   The unit functor $\et\colon \Id \rtarr \bP$ is induced by the object $e_1\in \sP(1)$, $\et(a) = (e_1, a)$.  The product functor $\mu\colon \bP\bP \rtarr \bP$ is induced by the structure maps $\ga$ of $\sP$, just as $1$-categorically, but here the associativity only holds up to a pseudoisomorphism.  

There are two variants, one general, in which the unit diagrams also only hold up to a pseudoisomorphism and one normal, in which the unit pseudoisomorphisms are identities.    With the latter, we have an alternative reduced monad $\bP$ in which the strict units allow the redefinition $\bP = \sP\otimes_{\LA} \sA^*$.  \autoref{thmP} still holds with the reduced variant.  On the topological $1$-category level, the reduced analog is far more useful (see \cite[section 4]{Rant1} for discussion), but that need not concern us here.

\subsection{The $2$-monads $\bE$ and $\bF$} 

We have analogs with $\sP$ replaced by $\sF$, but it is convenient to start with the subcategory $\sE$ of  effective maps in $\sF$.   In loose analogy with $\bP$, for a (covariant) functor 
$\sA \colon \SI \rtarr \mathbf{Cat}$, we define a  functor $\bE\sA\colon \SI \rtarr \mathbf{Cat}$ by
\begin{equation}\label{Edefn}
(\bE\sA)(n) = \sqcup_{n\geq 0} \sE(-,\bn) \otimes_{\SI} \sA = \sqcup_{n\geq 0}\sqcup_{m\geq 0}\sE(\bm,\bn)\times_{\SI_m} \sA(\bm)\\
\end{equation}
Here $\sE(-,\bn)$ is the contravariant functor $\sE\colon \SI \rtarr  \mathbf{Set}$ represented by $\bn$, but viewed as taking values in discrete categories.  The $n$th term in the coproduct is then a category.  (The $0$th component gives a base object.)  Post composition with permutations gives that $\bE\sA$ is a covariant functor $\SI \rtarr \mathbf{Cat}$.

As in \autoref{decompose}, $\sF$ is constructed from $\sE$ by adding in projections, as formalized by the definition of $\LA^{\perp}$ in \autoref{Fetal}. In contrast to $\sE$, each hom set $\sF(\bm,\bn)$ is based, hence we now work in based categories.  For a (covariant) functor  $\sA \colon \LA^{\perp} \rtarr \mathbf{Cat}_*$, we define a  functor 
$\bF\sA\colon \LA^{\perp} \rtarr \mathbf{Cat}_*$ by
\begin{equation}\label{Fdefn}
(\bF\sA)(n) = \vee_{n\geq 0} \sE(-,\bn) \otimes_{\LA^{\perp}} \sA = \vee_{n\geq 0}\vee_{m\geq 0}\sE(\bm,\bn)\times_{\SI_m} \sA(\bm)/(\sim),\\
\end{equation}
where $\sim$ comes from the projections in $\sF$.  

A key function of projections in $\sF$ is to facilitate comparison of categories with $\PI$-categories via the adjunction $(\bL,\bR)$ and its unit $\de$.    In view of the following lemma, projections are inessential to the constructions we are focusing on now. 

\begin{lem}\label{EinF}  The inclusion $\sE\subset \sF$ induces  an isomorphism $\bE\sA \iso \bF\sA$. 
\end{lem}
\begin{proof} Writing out the coequalizers that define the categorical tensor products, we see a commutative diagram
\[
\xymatrix@1{  
\sqcup_{m,n} \sE(\bm,\bn)\times \SI_m \times \sA(m) \ar@<.5ex> [r]^-{\com \times \id}  \ar@<-.5ex>[r]_-{\id\times \Psi} \ar[d] &  \sqcup_n\sqcup_m \sE(\bm,\bn)\times\sA(m)   \ar[r]  \ar[d] & \bE\sA \ar[d]\\ 
\vee_{\ell, m, n} \sF(\bm,\bn)\times \LA^{\perp}(\ell, \bm) \times \sA(\ell) \ar@<.5ex> [r]^-{\com\times \id}  \ar@<-.5ex>[r]_-{\id\times \Psi}  & \vee_{m} \sF(\bm,\bn)\times \sA(m)  \ar[r]  &  \bF\sA \\ }
\]
in which the vertical arrows are inclusions.    Writing a map $\ps\colon \ell \rtarr \bm$ in $\LA^{\perp}$ as the composite of a projection $\pi\colon \ell \rtarr \bk$ and a map $\si\colon \bk\rtarr \bm$ in $\sE$ as in 
\autoref{decompose}, $\si$ is in $\LA^{\perp}\cap \sE = \SI$, so that $k=m$ and $\si\in \SI_m$.  For any object or morphism $a\in \sA(m)$,  $(\si\pi, a) \sim (\si, \pi a)$, which implies that the right vertical arrow is a surjection.
\end{proof}

\begin{prop}\label{MonF}  With unit $\et$ and product $\mu$ induced by the identities and composition in $\sF$, $\bF$ is a $2$-monad in the category of $\LA^{\perp}$-categories.
\end{prop}
\begin{proof}  We define $\et\colon \sA \rtarr \bF \sA$ by use of the identity maps of the$\sF(n,n)$. 
We define $\mu\colon \bF\bF \rtarr \bF$ to be induced by composition. However, writing out $\bF\bF\sA$ explicitly, we see that we need compositions $\sF(\bp, \bq)(\bm,\bn)$ to be defined even when $n\neq p$.
We define it to be the base map $\bm\rtarr \ast \rtarr \bq$ and check that this makes sense of things.
 \end{proof}
 
 \begin{sch} While $\vee$ rather than $\sqcup$ was implicit in the general based context of 
 \cite[Section 4]{GMMO2}, the need for the last sentence of the proof just given was overlooked.
 \end{sch}
 
The proof of the following observation is a direct comparison, but it is also the specialization to $\sF$ of a general result about categories of operators.  
 
 \begin{prop}\label{thmF} The $2$-category  $\sF\text{-}\bf{PsAlg}$ is isomorphic to the $2$-category \linebreak $\bF\text{-}\bf{PsAlg}$ of pseudoalgebras (in 
 $\PI$-$\bf{Cat}$) over the $2$-monad $\bF$ and similarly and isomorphically for  
 $\sE\text{-}\bf{PsAlg}$ and $\bE\text{-}\bf{PsAlg}$ (in $\LA$-$\bf{Cat}$).
 \end{prop}
 
 \begin{warn}\label{warn} In sharp contrast to the focus of this paper,  the $\bF$-algebras $\bF\bR\sA$ 
 are not very special.  
 \end{warn}
 
 The warning gives one reason that $2$-monads are peripheral to this paper.  Nevertheless, it would be interesting to explore the implicit relation between Theorems \ref{thmP} and \ref{thmF} established by our \autoref{thm2}.

\bibliographystyle{plain}
\bibliography{references}

\end{document}